\documentclass[a4paper,11pt]{amsart}
\usepackage{amsfonts,stmaryrd,amssymb,amsmath,a4wide,verbatim}
\usepackage[british]{babel}
\usepackage{graphics}
\usepackage{epsfig,graphicx}
\usepackage[active]{srcltx}

\makeatletter
\@addtoreset{equation}{section}\makeatother

\newcommand{\scal}[2]{\left\langle #1 , #2\right\rangle}
\newcommand{\scalpoly}[2]{\left( #1 , #2\right)_{{\S}^n}}
\newcommand{\scalfonc}[2]{\left( #1 , #2\right)}
\newcommand{\normdpoly}[1]{\left\| #1\right\|_{\S^n}}
\newcommand{\normdfonc}[1]{\| #1\|_2}
\newcommand{\normdappli}[1]{\interleave #1\interleave}

\def\tr{{\rm tr}\,}

\newtheorem{theorem}{Theorem}[section]
\newtheorem{lemma}[theorem]{Lemma}
\newtheorem{proposition}[theorem]{Proposition}
\newtheorem{remark}[theorem]{Remark}

\def\R{\mathbb{R}}
\def\N{\mathbb{N}}
\def\S{\mathbb{S}}
\def\insm{\displaystyle\int_{M}}
\def\inssn{\int_{\S^n}}
\def\vol{dv}
\def\voleps{dv_{g_{\varepsilon}}}
\def\Vol{{\rm Vol}\,}
\def\xb{\overline{X}}

\def\la{\lambda_1}
\def\can{\text{can\ }}

\def\muks{\mu_k}

\def\hkm{{\mathcal H}^k(M)}
\def\hkr{{\mathcal H}^k(\R^{n+1})}
\def\coef{\fieps'^2+(1\pm\fieps)^2}
\def\coeff{\varphi'^2+(1+\varphi)^2}
\def\fieps{\varphi_{\varepsilon}}

\def\fiteps{\tilde{\varphi}_{\varepsilon}}
\def\mpeps{M_{\varepsilon}^+}
\def\mmeps{M_{\varepsilon}^-}
\def\meps{M_{\varepsilon}}

\def\ieps{I_{\varepsilon}}
\def\xeps{X_{\varepsilon}}
\def\Heps{H_{\varepsilon}}
\def\heps{h^{\pm}_{\varepsilon}}
\def\beps{B_{\varepsilon}}
\def\hepsi{h^{\pm}_{1,\varepsilon}}
\def\hepsd{h^{\pm}_{2,\varepsilon}}
\def\hepst{h^{\pm}_{3,\varepsilon}}

\begin{document}
\title[]{Hypersurfaces with small extrinsic radius or large $\lambda_1$ in Euclidean spaces}



\author[E. AUBRY, J.-F. GROSJEAN, J. ROTH]{Erwann AUBRY, Jean-Fran\c cois GROSJEAN, Julien ROTH}

\address[E. Aubry]{LJAD, Universit\'e de Nice Sophia-Antipolis, CNRS; 28 avenue Valrose, 06108 Nice, France}
\email{eaubry@unice.fr}

\address[J.-F. Grosjean]{Institut \'Elie Cartan (Math\'ematiques), Universit\'e Henri Poinca\-r\'e Nancy I, B.P. 239, F-54506 Vand\oe uvre-les-Nancy cedex, France}
\email{grosjean@iecn.u-nancy.fr}

\address[J. Roth]{LAMA, Universit\'e Paris-Est - Marne-la-Vall\'ee, 5 bd Descartes, Cit\'e Descartes, Champs-sur-Marne, F-77454 Marne-la-Vall\'ee }
\email{Julien.Roth@univ-mlv.fr}

\date{\today}

\begin{abstract} We prove that hypersurfaces of $\R^{n+1}$ which are almost extremal for the Reilly inequality on $\lambda_1$ and have $L^p$-bounded mean curvature ($p>n$) are Hausdorff close to a sphere, have almost constant mean curvature and have a spectrum which asymptotically contains the spectrum of the sphere. We prove the same result for the Hasanis-Koutroufiotis inequality on extrinsic radius. We also prove that when a supplementary $L^q$ bound on the second fundamental is assumed, the almost extremal manifolds are Lipschitz close to a sphere when $q>n$, but not necessarily diffeomorphic to a sphere when $q\leqslant n$.
\end{abstract}

\maketitle

\section{Introduction}
Sphere theorems in positive Ricci curvature are now a classical matter of study. The canonical sphere $(\S^n,can)$ is the only manifold with ${\rm Ric}\geqslant n{-}1$ which is extremal for the volume, the radius, the first non zero eigenvalue $\lambda_1$ on functions or the diameter. Moreover, it was proved in \cite{Co1,Co2,Ch-Co} that manifolds with ${\rm Ric}\geqslant n{-}1$ and volume close to $\Vol(\S^n,can)$ are diffeomorphic and Gromov-Hausdorff close to the sphere. This stability result was extended in \cite{Pe,Au}, where it is proved that manifolds with ${\rm Ric}\geqslant n-1$ have almost extremal volume if and only if they have almost extremal radius, if and only if they have almost extremal $\lambda_n$. Almost extremal diameter and almost extremal $\lambda_1$ are also equivalent when ${\rm Ric}\geqslant n{-}1$ (\cite{Cr,Il}), but, as shown in \cite{An,Ot}, it does not force the manifold to be diffeomorphic nor Gromov-Hausdorff close to $(\S^n,can)$.
In this paper, we study the stability of three optimal geometric inequalities involving the mean curvature of Euclidean hypersurfaces, and whose equality case characterizes the Euclidean spheres (see Inequalities \eqref{prim}, \eqref{rext} and \ref{lambda} below). More precisely we study the metric and spectral properties of the hypersurfaces which almost realize the equality case. It completes the results of \cite{colgros,roth}.

Let $X:(M^n,g)\to\R^{n+1}$ be a closed, connected, isometrically immersed $n$-manifold ($n\geqslant 2)$. The first geometric inequality we are interested in is the following
\begin{equation}\label{prim}
\|H\|_2 \|X-\xb\|_2\geqslant 1 
\end{equation}
where $\xb:=\frac{1}{\Vol M}\displaystyle\int_MX\vol$, $\Vol M$ is the volume of $(M^n,g)$, $H$ is the mean curvature of the immersion $X$ and $\|\cdot \|_p$ is the renormalized $L^p$-norm on $C^{\infty}(M)$ defined by $\|f\|_p^p=\frac{1}{\Vol M}\insm |f|^p\vol$.
Equality holds in \eqref{prim} if and only if $X(M)$ is a sphere of radius $\frac{1}{\|H\|_2}$ and center $\xb$ (see section \ref{prel}).
From \eqref{prim} we easily infer the Hasanis-Koutroufiotis inequality on extrinsic radius (i.e. the least radius of the balls of $\R^{n+1}$ which contains $X(M)$)
\begin{equation}\label{rext}
\|H\|_2 R_{ext}\geqslant1
\end{equation}
whose equality case also characterizes the sphere of radius $\frac{1}{\|H\|_2}$ and center $\overline{X}$.
The last inequality is the well-known Reilly inequality
\begin{equation}\label{lambda}
\la\leqslant n\|H\|_2^2
\end{equation}
Here also, the extremal hypersurfaces  are the spheres of radius  $\frac{1}{\|H\|_2}=\sqrt{\frac{n}{\lambda_1}}$. 
Let $p>2$ and $\varepsilon\in(0,1)$ be some reals. We will say that $M$ is almost extremal for Inequality \eqref{prim} when it satisfies the pinching
$$\displaylines{
(P_{p,\varepsilon})\hfill\|H\|_p\bigl\|X-\xb\bigr\|_2\leqslant 1+\varepsilon,\hfill
}$$
We will say that $M$ is almost extremal for Inequality \eqref{rext} when it satisfies the pinching
$$\displaylines{(R_{p,\varepsilon})\hfill \|H\|_p R_{ext}\leqslant1+\varepsilon\hfill} $$
We will say that $M$ is almost extremal for Inequality \eqref{lambda} when it satisfies the pinching
$$\displaylines{(\Lambda_{p,\varepsilon})\hfill(1+\varepsilon)\lambda_1\geqslant n\|H\|^2_p\hfill}$$
\begin{remark}
  It derives from the proof of the three above geometric inequalities, given in section \ref{prel}, that Pinching $(R_{p,\varepsilon})$ or Pinching $(\Lambda_{p,\varepsilon})$ imply Pinching $(P_{p,\varepsilon})$. For that reason, Theorems \ref{1}, \ref{maintheo}, \ref{2} below are stated for hypersuraces satisfying Pinching $(P_{p,\varepsilon})$ but are obviously valid for Pinching $(R_{p,\varepsilon})$ or Pinching $(\Lambda_{p,\varepsilon})$.
\end{remark}

Our first result is that, when $\|H\|_q$ is bounded, almost extremal manifolds for one of the three Inequalities \eqref{prim}, \eqref{rext} or \eqref{lambda} are Hausdorff close to an Euclidean sphere of radius $\frac{1}{\|H\|_2}$ and have almost constant mean curvature.

\begin{theorem}\label{1} Let $q>n$, $p>2$ and $A>0$ be some reals. There exist some po\-sitive functions $C=C(p,q,n,A)$ and $\alpha=\alpha(q,n)$ such that if $M$ satisfies $(P_{p,\varepsilon})$ and $\Vol M\|H\|_q^n\leqslant A$, then we have
\begin{align}\label{estiray}\Bigl\|\bigl|X-\xb\bigr|-\frac{1}{\|H\|_2}\Bigr\|_\infty\leqslant C \varepsilon^{\alpha}\frac{1}{\|H\|_2},\end{align}
and there exist some positive functions $C=C(p,q,r,n,A)$ and $\beta=\frac{\alpha(q-r)}{r(q-1)}$ so that
\begin{align}\label{pinchcm}\bigl\||H|-\|H\|_2\bigr\|^{~}_r\leqslant C\varepsilon^\beta\|H\|_2\ \ \ \ \ \ \ \mbox{for any }r\in[1,q).\end{align}
We assume moreover that $q>\max(4,n)$. For any \ $r>0$\  and\  $\eta>0$, there exists\  $\varepsilon_0=\varepsilon_0(p,q,n,A,r,\eta)>0$ such that if $M$ satisfies $(P_{p,\varepsilon})$ (for $\varepsilon\leqslant\varepsilon_0$) and $\Vol M\|H\|_q^n\leqslant A$ then for any $x\in S=\overline{X}+\frac{1}{\|H\|_2}\cdot\S^n$, we have
\begin{align}\label{passepartout}\Bigl|\frac{\Vol\bigl(B(x,\frac{r}{\|H\|_2})\cap X( M)\bigr)}{\Vol M}-\frac{\Vol\bigl(B(x,\frac{r}{\|H\|_2})\cap S\bigr)}{\Vol S}\Bigr|\leqslant \eta\frac{\Vol\bigl(B(x,\frac{r}{\|H\|_2})\cap S\bigr)}{\Vol S}\end{align}
where $B(x,r)$ is the Euclidean ball with center $x$ and radius $r$.
\end{theorem}

Theorem \ref{1} generalizes and improves the main results of \cite{colgros} and \cite{roth}, where only the pinchings $(R_{p,\varepsilon})$ and $(\Lambda_{p,\varepsilon})$ for $p\geqslant 4$ and $q=\infty$ were considered. The control on the mean curvature (Inequality \eqref{pinchcm}) and Inequality \eqref{passepartout} are new, even under a $L^\infty$ bound on the mean curvature. Note that \eqref{passepartout} says not only that $M$ goes near any point of the sphere $S$ (as was proven in \cite{colgros,roth}) but also that the density of $M$ near each point of $S$ is close to $\Vol M/\Vol S$.

\begin{remark}
  From Inequalities \eqref{estiray} and \eqref{passepartout} we infer that almost extremal hypersurfaces for one of the three geometric inequalities \eqref{prim}, \eqref{rext} or \eqref{lambda} converge in Hausdorff distance to a metric sphere of $\R^{n+1}$. As shown in Theorem \ref{ctrexple1}, there is no Gromov-Hausdorff convergence if we do not assume a good enough bound on the second fundamental form.
\end{remark}

\begin{remark}
  By Theorem \ref{1}, when $\Vol M\|H\|_q^n\leqslant A$ ($q>n$), Pinching  $(P_{p,\varepsilon})$ implies Pinching $(R_{p,\varepsilon'})$ for a constant $\varepsilon'=\varepsilon'(\varepsilon|A,p,q,n)$. In other words, Pinchings  $(P_{p,\varepsilon})$ and $(R_{p,\varepsilon})$ are equivalent (in bounded mean curvature) and are both implied by Pinching $(\Lambda_{p,\varepsilon})$. However, we will see in Theorem \ref{ctrexple1} that Pinching $(P_{p,\varepsilon})$ (or $(R_{p,\varepsilon})$) does not imply Pinching $(\Lambda_{p,\varepsilon})$.
\end{remark}

\begin{remark}
   The constant $C(p,q,n,A)$ tends to $\infty$ when $p\to 2$ or $q\to n$, but the same result can be proved with $\Vol M\|H\|_q^n\leqslant A$ replaced by $\Vol M\|H-\|H\|_2\|_n^n\leqslant A(n)$, where $A(n)$ is a universal constant depending only on the dimension $n$.
\end{remark}

Inequality \ref{estiray} follows from the following new pinching result on momenta.

\begin{theorem}\label{diambound}
Let $q>n$ be a real. There exists a constant $C=C(q,n)$ such that for any isometrically immersed hypersurface $M$ of $\R^{n+1}$, we have
$$\sup_M\Bigl||X-\overline{X}|-\|X-\overline{X}\|_2\Bigr|\leqslant  C\bigl(\Vol M\|H\|_q^n\bigr)^\gamma\|X-\overline{X}\|_2\Bigl(1-\frac{\|X-\overline{X}\|_1}{\|X-\overline{X}\|_2}\Bigr)^\frac{1}{2(1+n\gamma)}$$
where $\gamma=\frac{q}{2(q-n)}$.

In particular, this gives
$$\|X-\xb\|_\infty\leqslant C\bigl(\Vol M\|H\|_q^n\bigr)^\gamma\|X-\xb\|_2$$
\end{theorem}

Our next result shows that almost extremal hypersurfaces must satisfy strong spectral constraints. We denote $0=\mu_0<\mu_1<\cdots<\mu_i<\cdots$ the eigenvalues of the canonical sphere $\S^n$, $m_i$ the multiplicity of $\mu_i$ and $\displaystyle\sigma_k=\sum_{0\leqslant i\leqslant k} m_i$ (note that we have $\sigma_k=O(n^k)$ and $m_k=O(n^k)$). We also denote $0=\lambda_0(M)<\lambda_1(M)\leqslant\cdots\leqslant\lambda_i(M)\leqslant\cdots$ the eigenvalues of $M$ counted with multiplicities.

\begin{theorem}\label{maintheo}
Let $q>\max(n,4)$, $p>2$ and $A>0$ be some reals. There exist some positive functions $C=C(p,q,n,A)$ and $\alpha=\alpha(q,n)$ such that if $M$ satisfies $(P_{p,\varepsilon})$ and $\Vol M\|H\|_q^n\leqslant A$ then for any $k$ such that $2\sigma_kC^{2k}\leqslant\varepsilon^{-\alpha}$, the interval 
$$\bigl[(1-\varepsilon^\alpha\sqrt{m_k}C^{k})\|H\|_2^2\mu_k,(1+\varepsilon^\alpha\sqrt{m_k}C^{k})\|H\|_2^2\mu_k\bigr]$$
contains at least $m_k$ eigenvalues of $M$ counted with multiplicities.

Moreover, the previous intervals are disjoints and we get
$$\lambda_i(M)\leqslant\bigl(1+\varepsilon^\alpha\sqrt{m_k}C^{k}\bigr)\|H\|_2^2\lambda_i(\S^n)\ \ \ \ \ \ \mbox{ for any }i\leqslant \sigma_k-1,$$
and if $\lambda_{\sigma_k}(M)\geqslant\bigl(1+\varepsilon^\alpha\sqrt{m_k}C^k\bigr)\|H\|_2^2\mu_k$ then 
$$\bigl|\lambda_i(M)-\|H\|_2^2\lambda_i(\S^n)\bigr|\leqslant\varepsilon^\alpha\sqrt{m_k}C^k\|H\|_2^2\lambda_i(\S^n)\ \ \ \ \ \mbox{ for any }i\leqslant \sigma_k-1.$$
\end{theorem}

\begin{remark}
  In the particular case of extremal hypersurfaces for Pinching $(\Lambda_{p,\varepsilon})$, Theorem \ref{maintheo} implies that 
$$\frac{n\|H\|_p^2}{1+\varepsilon}\leqslant\lambda_1(M)\leqslant\cdots\leqslant\lambda_{n+1}(M)\leqslant\bigl(1+C(n)\varepsilon^\alpha\bigr)n\|H\|_2^2$$
and so we must have the $n+1$-first eigenvalues close to each other. Compare to positive Ricci curvature where $\lambda_n$ close to $n$ implies $\lambda_{n+1}$ close to $n$, but we can have only $k$ eigenvalues close to $n$ for any $k\leqslant n-1$ (see \cite{Au}).
\end{remark}

Note that Theorem \ref{maintheo} does not say that the spectrum of almost extremal hypersurfaces for I\-ne\-quality \eqref{prim} is close to the spectrum of an Euclidean sphere, but only that the spectrum of the sphere $S=\overline{X}+\frac{1}{\|X\|_2}\cdot\S^n$ asymptotically appears in the spectrum of $M$. Our next two results show that this inclusion is strict in general (we have normalized the mean curvature by $\|H\|_2=1$ for sake of simplicity and $E(x)$ stands for the integral part of $x$).

\begin{theorem}\label{ctrexple1}
For any integers $l,p$ there exists sequence of embedded hypersurfaces $(M_j)$ of $\R^{n+1}$ diffeomorphic to $p$ spheres $\S^n$ glued by connected sum along $l$ points, such that $\|H_j\|_{\infty}^{~}\leqslant C(n)$, $\|B_j\|_{n}^{~}\leqslant C(n)$, $\bigl\||X_j|-1\bigr\|_\infty\to0$, $\bigl\||H_j|-1\bigr\|_1\to 0$, and for any $\sigma\in\N$ we have
$$\lambda_\sigma(M_j)\to\lambda_{E(\frac{\sigma}{p})}(\S^n).$$
In particular, the $M_j$ have at least $p$ eigenvalues close to $0$ whereas its extrinsic radius is close to $1$.
\end{theorem}

\begin{theorem}\label{ctrexple3}
There exists sequence of immersed hypersurfaces $(M_j)$ of $\R^{n+1}$ diffeomorphic to $2$ spheres $\S^n$ glued by connected sum along $1$ great subsphere $S^{n-2}$, such that $\|H_j\|_{\infty}^{~}\leqslant C(n)$, $\|B_j\|_2^{~}\leqslant C(n)$, $\bigl\||X_j|-1\bigr\|_\infty\to0$, $\bigl\||H_j|-1\bigr\|_1\to 0$, and for any $\sigma\in\N$ we have
$$\lambda_\sigma(M_j)\to\lambda_{E(\frac{\sigma}{2})}(\S^{n,d}),$$
where $\S^{n,d}$ is the sphere $\S^n$ endowed with the singular metric, pulled-back of the canonical metric of $\S^n$ by the map $\pi:(y,z,r)\in\S^1\times\S^{n-2}\times[0,\frac{\pi}{2}]\mapsto(y^d,z,r)\in\S^1\times\S^{n-2}\times[0,\frac{\pi}{2}]$, where $\S^1\times\S^{n-2}\times[0,\frac{\pi}{2}]$ is identified with $\S^n\subset \R^{2}\times\R^{n-1}$ via the map $\Phi(y,z,r)=\bigl((\sin r) y,(\cos r) z\bigr)$.
Note that $\S^{n,d}$ has infinitely many eigenvalues that are not eigenvalues of $\S^n$.
\end{theorem}

\begin{remark}
  Theorem \ref{ctrexple1} shows that Pinching $(\Lambda_{p,\varepsilon})$ is not implied by Pinching $(P_{p,\varepsilon})$ nor Pinching $(R_{p,\varepsilon})$, even under an upper bound on $\|B\|_n$.
\end{remark} 

\begin{remark}\label{rem}
It also shows that almost extremal manifolds are not necessarily diffeomorphic nor Gromov-Hausdorff close to a sphere. We actually prove that the $(M_j)$ can be constructed by gluing spheres along great subspheres $S^{k_i}$ with $k_i\leqslant k\leqslant n-2$ and with $\|B_j\|_{n-k}\leqslant C(k,n)$ (see the last section of this article).
\end{remark}

In \cite{colgros} and \cite{roth} it has been proved that when the $L^{\infty}$-norm of the second fundamental form is bounded above, then almost extremal hypersurfaces are Lipschitz close to a sphere $S$ of radius $\frac{1}{\|H\|_2}$ (which implies closeness of the spectra). In view of Theorem \ref{ctrexple1}, we can wonder what stands when $\|B\|_q$ is bounded with $q>n$.

\begin{theorem}\label{2} Let $q>n$, $p>2$ and $A>0$ be some reals. There exist some po\-sitive functions $C=C(p,q,n,A)$ and $\alpha=\alpha(q,n)$ such that if $M$ satisfies $(P_{p,\varepsilon})$ and $\Vol M\|B\|_q^n\leqslant A$, then the map $$\begin{matrix}F&:&M&\longrightarrow&\frac{1}{\normdfonc{H}}\S^n\\
 & &x&\longmapsto&\frac{1}{\normdfonc{H}}\frac{X_x}{|X_x|}
\end{matrix}$$ is a diffeomorphism and satisfies $||dF(u)|^2-|u|^2|\leqslant C\varepsilon^{\alpha}|u|^2$ for any vector $u\in T M$.
\end{theorem}
\bigskip

The structure of the paper is as follows: after preliminaries on the geometric inequalities for hypersurfaces in Section \ref{prel}, we prove in Section \ref{upoter} a general bound on extrinsic radius that depends on integral norms of the mean curvature (see Theorem \ref{diambound}). We prove Inequality \eqref{estiray} in Section \ref{poie} and Inequality \eqref{pinchcm} in Section \ref{poip}. Theorem \ref{2} is proven in Section \ref{pot2}. Section \ref{Homog} is devoted to estimates on the trace on hypersurfaces of the homogeneous, harmonic polynomials of $\R^{n+1}$. These estimates are used in Section \ref{potm} to prove Theorem \ref{maintheo} and in section \ref{poipa} to prove Inequality \eqref{passepartout}. We end the paper in section \ref{se} by the constructions of Theorems \ref{ctrexple1} and \ref{ctrexple3}.

 Throughout the paper we adopt the notation that $C(p,q,n,A)$ is function greater than $1$ which depends on $p$, $q$, $n$, $A$. These functions  will always be of the form $C=D(p,q,n)A^{\beta(q,n)}$. But it eases the exposition to disregard the explicit nature of these functions. The convenience of this notation is that even though $C$ might change from line to line in a calculation it still maintains these basic features.

\section{Preliminaries}\label{prel}

Let $X:(M^n,g)\to\R^{n+1}$ be a closed, connected, isometrically immersed $n$-manifold ($n\geqslant 2)$. If $\nu$ denotes a local normal vector field of $M$ in $\R^{n+1}$, the second fundamental form of $(M^n,g)$ associated to $\nu$ is $B(\cdot\,,\cdot){=}\scal{\nabla^0_\cdot \nu}{\cdot}$ and the mean curvature is $H{=}(1/n)\tr(B)$, where $\nabla^0$ and $\langle\cdot\, ,\cdot\rangle$ are the Euclidean connection and inner product on $\R^{n+1}$.

Any function $F$ on $\R^{n+1}$ gives rise to a function $F\circ X$ on $M$ which, for more convenience, will be also denoted $F$ subsequently. An easy computation gives the formula
\begin{equation}
 \label{fondhess}
 \Delta F=nHdF(\nu)+\Delta^0F+\nabla^0dF(\nu,\nu),
\end{equation}
where $\Delta$ denotes the Laplace-Beltrami operator of $(M,g)$ and $\Delta^0$ is  the Laplace-Beltrami operator of $\R^{n+1}$. Applied to $F(x)=x_i$ or $F(x)=\langle x,x\rangle$, Formula \ref{fondhess} gives the following
\begin{align}\label{xi}\Delta X_i=nH\nu_i,\ \ \ \ \ \ \ \  \langle\Delta X,X\rangle=nH\langle\nu,X\rangle\end{align}
\begin{align}\label{hsiung}\frac{1}{2}\Delta |X|^2=nH\scal{\nu}{X}-n,\ \ \ \ \ \ \ \ \int_MH\langle \nu,X\rangle\vol=\Vol M\end{align}
These formulas are fundamental to control the geometry of hypersurfaces by their mean curvature.

\subsection{A rough bound on geometry}\label{rbog}

The integrated Hsiung formula \eqref{hsiung} and the Cauchy-Schwarz inequality give the following
\begin{align}\label{prim2}\int_M\frac{ H\langle \nu,X\rangle\vol}{\Vol M}=1\leqslant\|H\|_2\bigl\|X-\xb\bigr\|_2=\|H\|_2\inf_{u\in\R^{n+1}}\|X-u\|_2\end{align}
 This inequality $\|H\|_2\|X-\overline{X}\|_2\geqslant 1$ is optimal since $M$ satisfies 
$$\|H\|_2\bigl\|X-\xb\bigr\|_2=1$$
if and only if $M$ is a sphere of radius $\frac{1}{\|H\|_2}$ and center $\xb$. Indeed, in this case $X-\xb$ and $\nu$ are everywhere colinear, hence the differential of the function $|X-\xb|^2$ is zero on $M$. Equality \eqref{hsiung} then implies that $H$ is constant on $M$ equal to $|X-\xb|^{-1}$.

\subsection{Hasanis-Koutroufiotis inequality on extrinsic radius}

We set $R$ the extrinsic Radius of $M$, i.e. the least radius of the balls of $\R^{n+1}$ which contain $M$. Then Inequality \eqref{prim2} gives
\begin{align}\label{HK}
  \|H\|_2R_{ext}\geqslant\|H\|_2\|X-\overline{X}\|_2&=\|H\|_2\inf_{u\in\R^{n+1}}\|X-u\|_2&\nonumber\\
 \|H\|_2R_{ext}&\geqslant1
\end{align}
and when $R_{ext}=\frac{1}{\|H\|_2}$, we have equality in \eqref{prim2}, i.e. $M$ is a sphere of radius $\frac{1}{\|H\|_2}$.

\subsection{Reilly inequality on $\mathbf{\lambda_1}$}

We translate $M$ so that $\xb=0$. By the min-max principle and Equality \eqref{xi}, we have
$$\displaylines{\frac{\lambda_1}{n}\|X\|_2^2\leqslant\frac{\int_M\langle X,\Delta X\rangle\vol}{n\Vol M}=\int_M\frac{H\langle\nu,X\rangle\vol}{\Vol M}=1=\Bigl(\int_M\frac{H\langle\nu,X\rangle\vol}{\Vol M}\Bigr)^2}$$
where $\lambda_1$ is the first nonzero eigenvalue of $M$. Combined with Inequality \eqref{prim2}, we get the Reilly inequality
\begin{align}\label{R}
  \lambda_1\leqslant \frac{n}{\Vol M}\insm H^2\vol.
\end{align}
Here also, equality in the Reilly inequality gives equality in $\ref{prim2}$ and so it characterizes the sphere of radius $\frac{1}{\|H\|_2}=\|X\|_2=\sqrt{\frac{n}{\lambda_1}}$.


\section{Upper bound on the extrinsic radius}\label{upoter}

In this section we prove Theorem \ref{diambound}.

\begin{proof}
We translate $M$ such that $\xb=0$. We set $\varphi=\bigl||X|-\|X\|_2\bigr|$. We have $|d\varphi^{2\alpha}|\leqslant2\alpha\varphi^{2\alpha-1}$, hence, using the Sobolev inequality (see \cite{Mi-Si})
\begin{align}\label{simon}
\|f\|_\frac{n}{n-1}\leqslant K(n)(\Vol M)^\frac{1}{n}\bigl(\|df\|_1+\|Hf\|_1\bigr)
\end{align}
we get for any $\alpha\geqslant 1$
$$\displaylines{\|\varphi\|_\frac{2\alpha n}{n-1}^{2\alpha}\leqslant K(n)(\Vol M)^\frac{1}{n}\bigl(2\alpha\|\varphi\|_{2\alpha-1}^{2\alpha-1}+\|H\varphi^{2\alpha}\|_1\bigr)
\hfill\cr
\hfill\leqslant K(n)(\Vol M)^\frac{1}{n}\bigl(2\alpha\|\varphi\|_{2\alpha-1}^{2\alpha-1}+\|H\|_q\|\varphi\|_\frac{2\alpha q}{q-1}^{2\alpha}\bigr)\hfill\cr
\hfill\leqslant K(n)(\Vol M)^\frac{1}{n}\bigl(2\alpha\|\varphi\|_\frac{(2\alpha-1)q}{q-1}^{2\alpha-1}+\|H\|_q\|\varphi\|_\infty\|\varphi\|_\frac{(2\alpha-1) q}{q-1}^{2\alpha-1}\bigr)\cr
\hfill\leqslant K(n)(\Vol M)^\frac{1}{n}\bigl(2\alpha+\|H\|_q\|\varphi\|_\infty\bigr)\|\varphi\|_\frac{(2\alpha-1)q}{q-1}^{2\alpha-1}}$$
We set $\nu=\frac{n(q-1)}{(n-1)q}$ and $\alpha=a_p\frac{q-1}{2q}+\frac{1}{2}$, where $a_{p+1}=\nu a_p+\frac{n}{n-1}$ and $a_0=\frac{2q}{q-1}$ (i.e. $a_p=a_0\nu^p+\frac{\nu^p-1}{\nu-1}\frac{n}{n-1}$). The previous inequality gives
$$\Bigl(\frac{\|\varphi\|_{a_{p+1}}}{\|\varphi\|_\infty}\Bigr)^\frac{a_{p+1}}{\nu^{p+1}}\leqslant\Bigl(K(n)(\Vol M)^\frac{1}{n}\bigl(\frac{a_p\frac{q-1}{q}+1}{\|\varphi\|_\infty}+\|H\|_q\bigr)\Bigr)^\frac{n}{\nu^{p+1}(n-1)}\Bigl(\frac{\|\varphi\|_{a_p}}{\|\varphi\|_\infty}\Bigr)^\frac{a_p}{\nu^p}$$
Since $q>n$ then $\nu>1$ and $\frac{a_p}{\nu^p}$ converges to $a_0+\frac{qn}{q-n}$ and we have
$$\displaylines{1\leqslant\Bigl(\frac{\|\varphi\|_{a_0}}{\|\varphi\|_\infty}\Bigr)^{2}\prod_{i=0}^\infty\Bigl(2K(n)(\Vol M)^\frac{1}{n}a_i\bigl(\frac{1}{\|\varphi\|_\infty}+\|H\|_q\bigr)\Bigr)^\frac{1}{\nu^{i}}
\hfill\cr
\hfill=\Bigl(\frac{\|\varphi\|_{a_0}}{\|\varphi\|_\infty}\Bigr)^{2}\Bigl(\prod_{i=0}^{\infty}a_i^{\frac{1}{\nu^i}}\Bigr)\Bigl(2K(n)(\Vol M)^\frac{1}{n}\bigl(\frac{1}{\|\varphi\|_\infty}+\|H\|_q\bigr)\Bigr)^{\frac{\nu}{\nu-1}}\hfill\cr
\hfill=C(q,n)\Bigl(\frac{\|\varphi\|_{a_0}}{\|\varphi\|_\infty}\Bigr)^{2}\Bigl((\Vol M)^\frac{1}{n}\bigl(\frac{1}{\|\varphi\|_\infty}+\|H\|_q\bigr)\Bigr)^\frac{n(q-1)}{q-n}\hfill\cr
\leqslant C(q,n)\Bigl(\frac{\|\varphi\|_{2}}{\|\varphi\|_\infty}\Bigr)^\frac{2(q-1)}{q}\Bigl((\Vol M)^\frac{1}{n}\bigl(\frac{1}{\|\varphi\|_\infty}+\|H\|_q\bigr)\Bigr)^\frac{n(q-1)}{q-n}}$$
hence we have
$$\|\varphi\|_\infty\leqslant C(q,n)\Bigl((\Vol M)^\frac{1}{n}\bigl(\frac{1}{\|\varphi\|_\infty}+\|H\|_q\bigr)\Bigr)^\frac{nq}{2(q-n)}\|\varphi\|_2$$
We set $\gamma=\frac{nq}{2(q-n)}$. If $\|\varphi\|_\infty\geqslant\|H\|_q^{-\frac{\gamma}{1+\gamma}}\|\varphi\|_2^\frac{1}{1+\gamma}$ then we get the result since we have
$$\displaylines{\|\varphi\|_\infty\leqslant C(q,n)\Bigl((\Vol M)^\frac{1}{n}\bigl(\|H\|_q^{\frac{\gamma}{1+\gamma}}\|\varphi\|_2^\frac{-1}{1+\gamma}+\|H\|_q\bigr)\Bigr)^\gamma\|\varphi\|_2\cr
\leqslant C(q,n)\bigl((\Vol M)^\frac{1}{n}\|H\|_q\bigr)^\gamma\Bigl(\|X\|_2^\frac{1}{1+\gamma}+\|\varphi\|_2^\frac{1}{1+\gamma}\Bigr)^\gamma\|\varphi\|_2^\frac{1}{1+\gamma}}$$
where we have used that $\|H\|_q\|X\|_2\geqslant 1$. We infer the result from the equality $\|\varphi\|_2=\sqrt{2}\|X\|_2\bigl(1-\frac{\|X\|_1}{\|X\|_2}\bigr)^{1/2}$.
If $\|\varphi\|_\infty\leqslant\|H\|_q^{-\frac{\gamma}{1+\gamma}}\|\varphi\|_2^\frac{1}{1+\gamma}$, we get immediately the desired inequality of the Theorem from the above expression of $\|\varphi\|_2$ and the fact that $\|H\|_q\|X\|_2\geqslant 1$.
\end{proof}

\section{Proof of Inequality \eqref{estiray}}\label{poie}

Let $M$ be an isometrically immersed hypersurface of $\R^{n+1}$. We can, up to translation, assume that $\int_M X\vol=0$.
By the H\"{o}lder inequality and Pinching $(P_{p,\varepsilon})$, we have $\|H\|_p\|X\|_2\leqslant(1+\varepsilon)\leqslant(1+\varepsilon)\|H\|_p\|X\|_\frac{p}{p-1}\leqslant(1+\varepsilon)\|H\|_p\|X\|_1^{1-\frac{2}{p}}\|X\|_2^\frac{2}{p}$, hence
$$1-\frac{\|X\|_1}{\|X\|_2}\leqslant\bigl((1+\varepsilon)^\frac{p}{p-2}-1\bigr)\leqslant\frac{p}{p-2}2^\frac{2}{p-2}\varepsilon$$

On the other hand applying Inequality \eqref{simon} to $f=1$ we get \begin{align}\label{minorA}1\leqslant K(n)(\Vol M)^\frac{1}{n}\|H\|_1\end{align}
And combining the two above inequalities with Theorem \ref{diambound} and  $1\leqslant\|H\|_2\|X\|_2\leqslant1+\varepsilon$ we get \eqref{estiray}. More precisely we have $\displaystyle\bigl\||X|-\frac{1}{\normdfonc{H}}\bigr\|_{\infty}\leqslant\frac{C(p,q,n)A^{\gamma/n}}{\normdfonc{H}}\varepsilon^{\alpha(q,n)}$.

\begin{remark}\label{majorx} Combining \eqref{minorA} with Inequality \eqref{estiray} we get
\begin{align}\label{ninfx}\|X\|_{\infty}\leqslant  C\bigl(p,q,n\bigr)A^{\gamma/n}(\Vol M)^{1/n}\end{align}
\end{remark}

\begin{lemma}\label{nldz} For any $0<\varepsilon<1$ if $(P_{p,\varepsilon})$ is satisfied, then there exist some positive functions $C(p,q,n)$, $\alpha(q,n)$ and $\beta(q,n)$ so that the vector field $Z=\nu-HX$ satisfies
\begin{align}\label{estizd}\|Z\|_r\leqslant C(p,q,n)(1+A)^{\beta} \varepsilon^\frac{\alpha(q-r)}{r(q-2)}\ \ \ \ \mbox{for any }r\in[2,q).\end{align}
\end{lemma}

\begin{proof} By the H\"older inequality we have for any $r\in[2,q)$
$$\displaylines{\|Z\|_r\leqslant\|Z\|_q^\frac{q(r-2)}{r(q-2)}\|Z\|_2^\frac{2(q-r)}{r(q-2)}\leqslant(1+\|X\|_{\infty}\|H\|_q)^{\frac{q(r-2)}{r(q-2)}}\|Z\|_2^{\frac{2(q-r)}{r(q-2)}}\cr
\leqslant(1+\|X\|_{\infty}\|H\|_q)^{\frac{q}{q-2}}\|Z\|_2^{\frac{2(q-r)}{r(q-2)}}}$$
By remark \ref{majorx}, we have $\|H\|_q\|X\|_{\infty}\leqslant C(p,q,n)A^\frac{\gamma+1}{n}$. Then
$$\|Z\|_r\leqslant C(p,q,n)A^{\beta}\|Z\|_2^{\frac{2(q-r)}{r(q-2)}}$$
Moreover by integrating the Hsiung-Minkowsky formula (\ref{hsiung}) we have
\begin{align*}\|Z\|_2^2&=1-\frac{2}{\Vol M}\insm H\scal{\nu}{X}\vol+\|HX\|_2^2\leqslant-1+\|H\|^2_2\|X\|_\infty^2.
\end{align*}
which, by Inequality \eqref{estiray}, gives
$\|Z\|^2_2\leqslant C(p,q,n)A^{\beta(q,n)}\varepsilon^{\alpha(q,n)}$.
\end{proof}

\section{Proof of Inequality \eqref{pinchcm}}\label{poip}

Since we have $1=\frac{1}{\Vol M}\int_MH\langle X,\nu\rangle\vol\leqslant\|H\|_2\|\langle X,\nu\rangle\|_2$, Inequality $(P_{p,\varepsilon})$ gives us
$$\displaylines{
\hfill\|X\|_2\leqslant(1+\varepsilon)\|\langle X,\nu\rangle\|_2,\hfill1\leqslant\|H\|_2\|X\|_2\leqslant1+\varepsilon,
\hfill}$$
and so
$$\displaylines{\hfill
\|X-\langle X,\nu\rangle\nu\|_2^{~}\leqslant\sqrt{3\varepsilon}\,\|X\|_2^{~},\hfill
\|X-\frac{H\nu}{\|H\|_2^2}\|_2^{~}=\sqrt{\|X\|_2^2-\|H\|_2^{-2}}\leqslant\sqrt{3\varepsilon}\,\|X\|_2^{~}.\hfill}$$
By Inequalities \eqref{estiray}, this gives
\begin{align*}\bigl\|H^2-\|H\|_2^2\bigr\|_1&\leqslant\bigl\|H^2-|X|^2\|H\|_2^4\bigr\|_1+\bigl\||X|^2\|H\|_2^4-\|H\|_2^2\bigr\|_1\\
&=\|H\|_2^4\Bigl(\bigl\|\frac{H^2}{\|H\|_2^4}-|X|^2\bigr\|_1+\bigl\||X|^2-\frac{1}{\|H\|_2^2}\bigr\|_1\Bigr)\\
&\leqslant\|H\|_2^4\Bigl(\bigl\|\frac{H\nu}{\|H\|_2^2}+X\bigr\|_2\bigl\|X-\frac{H\nu}{\|H\|_2^2}\bigr\|_2+\frac{CA^{\gamma/n}\varepsilon^{\alpha}}{\normdfonc{H}}\left(\normdfonc{X}+\frac{1}{\normdfonc{H}}\right)\Bigr)\\
&\leqslant\|H\|_2^4\Bigl(\sqrt{3\varepsilon}\|X\|_2(\|\frac{H\nu}{\|H\|_2^2}\|_2+\|X\|_2)+CA^{\gamma/n}\varepsilon^\alpha\frac{1}{\|H\|^2_2}\Bigr)\\
&\leqslant CA^{\gamma/n}\varepsilon^\alpha\|H\|_2^2
\end{align*}
Hence\  we\  have\ \  $\bigl\||H|-\|H\|_2\bigr\|_1\leqslant\frac{\|H^2-\|H\|_2^2\|_1^{~}}{\|H\|_2}\leqslant CA^{\gamma/n}\varepsilon^\alpha\|H\|_2$.
Moreover\ \  we\  have\ \  $\bigl\||H|-\|H\|_2\bigr\|_q\leqslant 2\|H\|_q\leqslant 2K(n)\|H\|_2(\Vol M)^\frac{1}{n}\|H\|_q$. Hence by the H\"older inequality, for any $r\in[1,q)$ we have
\begin{align*}\bigl\||H|-\|H\|_2\bigr\|_r&\leqslant\bigl(\bigl\||H|-\|H\|_2\bigr\|_1\bigr)^\frac{q-r}{r(q-1)}\bigl(\bigl\||H|-\|H\|_2\bigr\|_{q}\bigr)^\frac{q(r-1)}{r(q-1)}\\
&\leqslant C(p,q,r,n)A^{\beta(q,r,n)}\varepsilon^{\frac{\alpha(q-r)}{r(q-1)}}\|H\|_2\end{align*}

\section{Proof of the theorem \ref{2}}\label{pot2}

Let $u\in TM$ be a unit vector and put $\psi=|X^\top|$ where $X^\top$ is the tangential projection of $X$ on $TM$. For $\varepsilon$ small enough we have from \eqref{estiray}  $|X|\geqslant\frac{1}{2\normdfonc{H}}$ and then the application $F$ is well defined. We have $dF(u)=\frac{1}{\|H\|_2|X|}\bigl(u-\frac{\langle X,u\rangle}{|X|^2}X\bigr)$ (see \cite{colgros}), hence for any $\alpha\geqslant 1$
\begin{align}\label{quasisometry}||dF_x(u)|^2-1|&\leqslant\frac{1}{|X|^2}\left|\frac{1}{\normdfonc{H}^2}-|X|^2\right|+\frac{1}{\normdfonc{H}^2}\frac{\scal{u}{X}^2}{|X|^4}\\
&\leqslant\frac{C\varepsilon^{\alpha}}{|X|^2\normdfonc{H}^2}+\frac{\|\psi\|_{\infty}^2}{\normdfonc{H}^2|X|^4}\notag
\end{align}
Now an easy computation using  \ref{estiray} shows that $|d\psi|\leqslant|\scal{X}{\nu}B-g|\leqslant\frac{CA^{\beta}}{\normdfonc{H}}|B|+n$. Now using the Sobolev inequality \ref{simon} and the fact that $\gamma_n\leqslant(\Vol M)^{1/n}\normdfonc{H}\leqslant(\Vol M)^{1/n}\|H\|_q\leqslant(\Vol M)^{1/n}\|B\|_q\leqslant A^{1/n}$ (see \ref{minorA}), we have
\begin{align*}\|\psi\|_{\frac{2\alpha n}{n-1}}^{2\alpha}&\leqslant K(n)(\Vol M)^{1/n}\Bigl(2\alpha\frac{CA^{\beta}}{\normdfonc{H}}\|B\|_q+2\alpha n+\|H\|_q\|\psi\|_{\infty}\Bigr)\|\psi\|_\frac{(2\alpha-1)q}{q-1}^{2\alpha-1}\\
&\leqslant K(n)\left(2\alpha CA^{\beta}(\Vol M)^{1/n}+A^{1/n}\|\psi\|_{\infty}\right)\|\psi\|_\frac{(2\alpha-1)q}{q-1}^{2\alpha-1}
\end{align*}
And similarly to the proof of the theorem \ref{diambound} we obtain
$$\|\psi\|_{\infty}\leqslant C(q,n)\bigl(\frac{(\Vol M)^{1/n}CA^{\beta}}{\|\psi\|_{\infty}}+A^{1/n}\bigr)^{\gamma}\normdfonc{\psi}$$
And using the fact that $\|\psi\|_{\infty}\normdfonc{H}\leqslant\|X\|_{\infty}\normdfonc{H}\leqslant 1+C$ and $A\geqslant\gamma_n$ we get $\|\psi\|_{\infty}\leqslant CA^{\beta}\left(\frac{1}{\|\psi\|_{\infty}\normdfonc{H}}\right)^{\gamma}\normdfonc{\psi}$ that is
$$\|\psi\|_{\infty}^{\gamma+1}\leqslant CA^{\beta}\frac{\normdfonc{\psi}}{\normdfonc{H}^{\gamma}}$$
Now since $\normdfonc{\psi}=\normdfonc{X-\scal{X}{\nu}\nu}\leqslant\sqrt{3\varepsilon}\normdfonc{X}$ and $\normdfonc{H}\normdfonc{X}\leqslant 1+\varepsilon$ we deduce that $\|\psi\|_{\infty}\leqslant\frac{CA^{\beta}}{\normdfonc{H}}\varepsilon^{\alpha(q,n)}$. And reporting this in \eqref{quasisometry} and using \eqref{estiray} with the fact that $|X|\geqslant\frac{1}{2\normdfonc{H}}$ we get $||dF_x(u)|^2-1|\leqslant CA^{\beta}\varepsilon^{\alpha(q,n)}$.

\section{Homogeneous, harmonic polynomials of degree $k$}\label{Homog}

Let $\hkr$ be the space of homogeneous, harmonic polynomials of degree $k$ on $\R^{n+1}$. Note that $\hkr$ induces on $\S^n$ the spaces of eigenfunctions of $\Delta^{\S^n}$ associated to the eigenvalues $\muks:=k(n+k-1)$ with  multiplicity $m_k:=\begin{pmatrix}n+k-1\\ k\end{pmatrix}\displaystyle\frac{n+2k-1}{n+k-1}$ (see \cite{BGM}). 

On the space $\hkr$, we define the following inner product
$$\scalpoly{P}{Q}:=\frac{1}{\Vol\S^n}\inssn PQ dv_{\can},$$
where $dv_{can}$ denotes the element volume of the sphere with its standard metric. On the other hand the inner product on $M$ will be defined by 
$$\scalfonc{f}{g}=\insm \frac{fg\vol}{\Vol M}\ \ \ \mbox{ for }f,g\in C^{\infty}(M).$$

In this section we give some estimates on harmonic homogeneous polynomials needed subsequently. We set $(P_1,\cdots,P_{m_k})$ an arbitrary orthonormal basis of $\hkr$.
Remind that for any $P\in\hkr$ and any  $Y\in\R^{n+1}$, we have $dP(X)=kP(X)$ and $\nabla^0 dP(X,Y)=(k-1)dP(Y)$.

\begin{lemma}\label{Pcarre}
For any $x\in\R^{n+1}$, we have ${\displaystyle\sum_{i=1}^{m_k}}P_i^2(x)=m_k|x|^{2k}$.
\end{lemma}

\begin{proof} For any $x\in\S^n$, $Q_x(P)=P^2(x)$ is a quadratic form on $\hkr$ whose trace is given by $\sum_{i=1}^{m_k}P_i^2(x)$. Since for any $x'\in\S^n$ and any $O\in O_{n+1}$ such that $x'=Ox$ we have $Q_{x'}(P)=Q_x(P\circ O)$ and since $P\mapsto P\circ O$ is an isometry of $\hkr$, we have $ \sum_{i=1}^{m_k}P_i^2(x)=\tr(Q_x)=\sum_{i=1}^{m_k}P_i^2(x')=\tr( Q_{x'})$. Now
$$\sum_{i=1}^{m_k}\frac{1}{\Vol\S^n}\inssn P_i^2(x)\vol=m_k=\frac{1}{\Vol\S^n}\inssn\left(\sum_{i=1}^{m_k}P_i^2(x)\right)\vol$$
and so $\sum_{i=1}^{m_k}P_i^2(x)=m_k$. We conclude by homogeneity of the $P_i$.
\end{proof}

As an immediate consequence, we have the following lemma.

\begin{lemma}\label{grad}
For any $x,u\in\R^{n+1}$, we have
$$\sum_{i=1}^{m_k}\bigl(d_xP_i(u)\bigr)^2=m_k\left(\frac{\muks}{n}|x|^{2(k-1)}|u|^2+\left(k^2-\frac{\muks}{n}\right)\langle u,x\rangle^2|x|^{2(k-2)}\right).$$
\end{lemma}

\begin{proof} Let $x\in\S^n$ and $u\in\S^n$  so that $\langle u,x\rangle=0$. Once again the quadratic forms $Q_{x,u}(P)=\bigl(d_x P(u)\bigr)^2$ are conjugate (since $O_{n+1}$ acts transitively on orthonormal couples) and so $\displaystyle\sum_{i=1}^{m_k}\bigl(d_xP_i(u)\bigr)^2$ does not depend on $u\in x^{\perp}$ nor on $x\in\S^n$. By choosing an orthonormal basis $(u_j)_{1\leqslant j\leqslant n}$ of $x^\perp$, we obtain that
\begin{align*}\sum_{i=1}^{m_k}\bigl(d_x P_i(u)\bigr)^2&=\frac{1}{n}\sum_{i=1}^{m_k}\sum_{j=1}^n\bigl(d_xP_i(u_j)\bigr)^2=\frac{1}{n\Vol\S^n}\int_{\S^n}\sum_{i=1}^{m_k}|\nabla^{\S^n}P_i|^2\\
&=\frac{1}{n\Vol\S^n}\int_{\S^n}\sum_{i=1}^{m_k}P_i\Delta^{\S^n}P_i=\frac{m_k\muks}{n}
\end{align*}
Now suppose that $u\in\R^{n+1}$. Then $u=v+\langle u,x\rangle x$, where $v=u-\langle u,x\rangle x$, and we have
\begin{align*}\sum_{i=1}^{m_k}\bigl(d_xP_i(u)\bigr)^2&=\sum_{i=1}^{m_k}\bigl(d_xP_i(v)+k\langle u,x\rangle P_i(x)\bigr)^2\\
&=\sum_{i=1}^{m_k}\bigl(d_xP_i(v)\bigr)^2+2k\langle u,x\rangle\sum_{i=1}^{m_k}d_xP_i(v)P_i(x)+m_k\langle u,x\rangle^2 k^2\\
&=\frac{m_k\muks}{n}|v|^2+m_k\langle u,x\rangle^2 k^2=m_k\left(\frac{\muks}{n}|u|^2+\left(k^2-\frac{\muks}{n}\right)\langle u,x\rangle^2 \right),
\end{align*}
where we derived the equality in Lemma \ref{Pcarre} to make $\displaystyle\sum_{i=1}^{m_k}d_xP_i(v)P_i(x)$ disappear. We conclude by homogeneity of $P_i$.
\end{proof}

\begin{lemma}\label{Hess} For any $x\in\R^{n+1}$, we have ${\displaystyle\sum_{i=1}^{m_k}}|\nabla^0dP_i(x)|^2=m_k\alpha_{n,k}|x|^{2(k-2)}$, where $\alpha_{n,k}=(k-1)(k^2+\mu_k)(n+2k-3)\leqslant C(n)k^4$.
\end{lemma}

\begin{proof} The Bochner equality gives
$$\displaylines{
\sum_{i=1}^{m_k}|\nabla^0dP_i(x)|^2=\sum_{i=1}^{m_k}\left(\langle d\Delta^0 P_i,dP_i\rangle-\frac{1}{2}\Delta^0\bigl|dP_i\bigr|^2\right)\hfill\cr
=-\frac{1}{2}m_k\bigl(k^2+\mu_k\bigr)\Delta^0|X|^{2k-2}=m_k\alpha_{n,k}|X|^{2k-4}
}$$
\end{proof}

Let $\hkm=\{P\circ X\ ,\ P\in\hkr\}$ be the space of functions induced on $M$ by $\hkr$. We will identify $P$ and $P\circ X$ subsequently. There is no ambiguity since we have

\begin{lemma} Let $M^n$ be a compact manifold immersed by $X$ in $\R^{n+1}$ and let $(P_1,\ldots,P_m)$ be a linearly independent set of homogeneous polynoms of degree $k$ on $\R^{n+1}$. Then the set $(P_1\circ X,\ldots,P_m\circ X)$ is also linearly independent.
\end{lemma}

\begin{proof} Any homogeneous polynomial $P$ which is zero on $M$ is zero on the cone $\R^+{\cdot} M$. Since $M$ is compact there exists a point $x\in M$ so that $X_x\notin T_xM$ and so $\R^+{\cdot}M$ has non empty interior. Hence $P\circ X=0$ implies $P=0$.
\end{proof}

Formula \eqref{fondhess} implies
\begin{equation}\label{laplap}\Delta P=\muks H^2 P+(n+2k-2)HdP(Z)+\nabla^0 dP(Z,Z)\end{equation}
In order to estimate $\Delta P$, we define two linear maps 
$$
\begin{array}{rcl}
V_k^{\star}: \hkr & \longrightarrow & C^{\infty}(M)\\
P & \longmapsto & dP(V)
\end{array}
$$
and
$$
\begin{array}{rcl}
(V,W)_k^{\star}  :  \hkr & \longrightarrow & C^{\infty}(M)\\
P & \longmapsto & \nabla^0 dP(V,W)
\end{array}$$
where $V,W\in\Gamma(M)$ are vector fields.

If $L :\hkr \longrightarrow C^{\infty}(M)$ is a linear map, we set 
$$\normdappli{L}^2=\sum_{i=1}^{m_k}\normdfonc{L(P_i)}^2,$$
where $(P_1,\cdots,P_{m_k})$ is an orthonormal basis of $(\hkr,\normdpoly{\ .\ })$.
\begin{remark}\label{remark}
 For any $P\in\hkr$, we have $\|L(P)\|_2^2\leqslant\normdappli{L}^2\|P\|_{\S^n}^2$.
\end{remark}

We now give some estimates on $Z_k^{\star}$, $(HZ)_k^{\star}$ and $(Z,Z)_k^{\star}$.

\begin{lemma}\label{estim} We have
\begin{align}\label{normz}\normdappli{Z_k^{\star}}^2\leqslant  \frac{m_kk^2}{\Vol M}\insm|X|^{2(k-1)}|Z|^2\vol
\end{align}
\begin{align}\label{normhz}\normdappli{(HZ)_k^{\star}}^2\leqslant  \frac{m_kk^2}{\Vol M}\insm |X|^{2(k-1)}H^2|Z|^2\vol
\end{align}
\begin{align}\label{normzz}\normdappli{(Z,Z)_k^{\star}}^2\leqslant  \frac{m_k\alpha_{k,n}}{\Vol M}\insm |X|^{2(k-2)}|Z|^4\vol
\end{align}
\end{lemma}

\begin{proof} Let $(P_1,\cdots,P_{m_k})$  be an orthonormal basis of $\hkr$. By Lemma \ref{grad} we have
\begin{align*}\normdappli{Z_k^{\star}}^2&=\sum_{i=1}^{m_k}\normdfonc{dP_i(Z)}^2\leqslant \frac{m_kk^2}{\Vol M}\insm|X|^{2(k-1)}|Z|^2\vol\end{align*}
and
\begin{align*}\normdappli{(HZ)_k^{\star}}^2&=\sum_{i=1}^{m_k}\normdfonc{HdP_i(Z)}^2\leqslant \frac{m_kk^2}{\Vol M}\insm |X|^{2(k-1)}H^2|Z|^2\vol\end{align*}
By Lemma \ref{Hess}, we have
\begin{align*}\normdappli{(Z,Z)_k^{\star}}^2&=\sum_{i=1}^{m_k}\normdfonc{\nabla^0 dP_i(Z,Z)}^2\leqslant \frac{m_k\alpha_{k,n}}{\Vol M}\insm |X|^{2(k-2)}|Z|^4\vol,\end{align*}
which ends the proof.
\end{proof}

\begin{lemma}\label{Ppresquortho}  Let $q>n$ and $A>0$ be some reals. There exist a constants $C=C(q,n)$ and $\beta(q,n)$ such that for any isometrically immersed hypersurface $M$ of $\R^{n+1}$ which satisfies $\Vol M\|H\|_q^n\leqslant A$ and any $P\in\hkm$, we have
$$\Bigl|\|H\|_2^{2k}\normdfonc{P}^2-\normdpoly{P}^2\Bigr|\leqslant D\sigma_k(CA^{\beta})^{2k-1}\normdpoly{P}^2$$
where $D=\bigl\|H^2-\|H\|_2^{2}\bigr\|_1\|X\|^{2}_\infty+\bigl\|HZ\bigr\|_2\|X\|_\infty+\|Z\|_2^2+\bigl\|Z\bigr\|_4^2$.
\end{lemma}

\begin{proof} For any $P\in\hkm$ we have
\begin{align*}\|\nabla^0 P\|_2^2&=\|dP(\nu)\|_2^2+\|dP\|_2^2\\
&=\normdfonc{dP(Z)}^2+k^2\|HP\|_2^2+\frac{1}{\Vol M}\insm\bigl(2k H dP(Z)P+P\Delta P\bigr)\vol
\end{align*}
and from \eqref{laplap} we get
\begin{align*}\|\nabla^0 P\|_2^2=&\normdfonc{dP(Z)}^2+\frac{1}{\Vol M}\insm\bigl(P\nabla^0 dP(Z,Z)+(n+4k-2)H dP(Z)P\bigr)\vol\\
&+(\muks+k^2)\|HP\|_2^2\\
=&\frac{1}{\Vol M}\insm\Bigl((\muks+k^2)\bigl(H^2-\|H\|_2^2\bigr)P^2+(n+4k-2)HdP(Z)P\Bigr)\vol\\
&+\frac{1}{\Vol M}\insm P\nabla^0 dP(Z,Z)\vol+(\muks+k^2)\|H\|_2^2\|P\|_2^2+\normdfonc{dP(Z)}^2
\end{align*}
Now we have
\begin{align}\label{normgrad0}\normdpoly{\nabla^0 P}^2=\normdpoly{\nabla^{\S^n}P}^2+k^2\normdpoly{ P}^2=(\muks+k^2)\normdpoly{P}^2\end{align}
Hence 
$$\displaylines{
\|H\|_2^{2k-2}\normdfonc{\nabla^0 P}^2-\normdpoly{\nabla^0 P}^2=(\muks+k^2)\bigl(\|H\|_2^{2k}\normdfonc{P}^2-\normdpoly{P}^2\bigr)+\|H\|_2^{2k-2}\normdfonc{dP(Z)}^2\hfill\cr
+\frac{\|H\|_2^{2k-2}}{\Vol M}\insm P\Bigl((\muks+k^2)\bigl(H^2-\|H\|_2^2\bigr)P +H (n+4k-2)dP(Z)+\nabla^0 dP(Z,Z)\bigr)\vol\hfill}$$
Which gives
\begin{align}\label{intermediaire}\Bigl|\|H\|_2^{2k}\normdfonc{P}^2&-\normdpoly{P}^2\Bigr|\leqslant\frac{1}{\muks+k^2}\Bigl|\|H\|_2^{2k-2}\normdfonc{\nabla^0 P}^2-\normdpoly{\nabla^0 P}^2\Bigr|\\
&+\frac{\|H\|_2^{2k-2}}{\muks+k^2}\Bigl( (n+4k-2)\bigl|\scalfonc{(HZ)_k^{\star}P}{P}\bigr|+\normdfonc{Z_k^{\star}(P)}^2+\bigl|\scalfonc{(Z,Z)_k^{\star}P}{P}\bigr|\Bigr)\notag\\
&+\frac{\|H\|_2^{2k-2}}{\Vol M}\insm\bigl|H^2-\|H\|_2^{2}\bigr|P^2\vol\notag
\notag\end{align}
Note that, by Lemma \ref{Pcarre} and Remark \ref{remark}, we have
\begin{align*}\frac{1}{\Vol M}\insm \bigl| H^2-\|H\|_2^{2}\bigr|P^2\vol&\leqslant\frac{\normdpoly{P}^2}{\Vol M}\insm\bigl| H^2-\|H\|_2^{2}\bigr|\sum_{i=1}^{m_k}P_i^2\vol\\
&=\frac{m_k\normdpoly{P}^2}{\Vol M}\insm\bigl| H^2-\|H\|_2^{2}\bigr||X|^{2k}\vol\\
&\leqslant m_k \|X\|_{\infty}^{2k}\|H^2-\|H\|_2^2\|_1\normdpoly{P}^2
\end{align*}
which, combined with \eqref{intermediaire}, gives 
\begin{align*}&\Bigl|\|H\|_2^{2k}\normdfonc{P}^2-\normdpoly{P}^2\Bigr|\leqslant\frac{1}{\muks+k^2}\Bigl|\|H\|_2^{2k-2}\normdfonc{\nabla^0 P}^2-\normdpoly{\nabla^0 P}^2\Bigr|\\
&+m_k\|H\|_2^{2k-2}\|X\|_{\infty}^{2k}\|H^2-\|H\|_2^2\|_1\normdpoly{P}^2\\
&+\frac{\|H\|_2^{2k-2}\normdpoly{P}}{\muks+k^2}\Bigl((n+4k-2)\normdappli{(HZ)_k^{\star}}\normdfonc{P}+\normdappli{Z_k^{\star}}^2\normdpoly{P}+\normdappli{(Z,Z)_k^{\star}}\normdfonc{P}\Bigr)
\end{align*}
Now, as above, we have
\begin{align}\label{majorpfoncppoly}\normdfonc{P}\leqslant\sqrt{\sum_i\|P_i\|_2^2}\normdpoly{P} \leqslant \sqrt{\frac{m_k}{\Vol M}\int_M|X|^{2k}\vol}\normdpoly{P}\end{align}
and from Lemma \ref{estim}, we get

\begin{align*}\frac{\|H\|_2^{2k-2}\normdpoly{P}}{\muks+k^2}\Bigl((n+4k-2)\normdappli{(HZ)_k^{\star}}\normdfonc{P}+\normdappli{Z_k^{\star}}^2\normdpoly{P}+\normdappli{(Z,Z)_k^{\star}}\normdfonc{P}\Bigr)\\
\leqslant C(n)m_k(\|H\|_2\|X\|_{\infty})^{2k-2}(\|X\|_{\infty}\|HZ\|_2+\|Z\|_2^2+\|Z\|_4^2)\|P\|_{\S^n}^2
\end{align*}
and
\begin{align*}\Bigl|\|H\|_2^{2k}\normdfonc{P}^2-\normdpoly{P}^2\Bigr|\leqslant&\frac{1}{\muks+k^2}\Bigl|\|H\|_2^{2k-2}\normdfonc{\nabla^0 P}^2-\normdpoly{\nabla^0 P}^2\Bigr|\\
&+\bigl(\|X\|_\infty\|H\|_2\bigr)^{2k-2}m_kC(n)D\normdpoly{P}^2\end{align*}
with
$$\displaylines{D:=\left(\bigl\|H^2-\|H\|_2^{2}\bigr\|_1\|X\|^{2}_\infty+\bigl\|HZ\bigr\|_2\|X\|_\infty+\|Z\|_2^2+\bigl\|Z\bigr\|_4^2\right)
}$$
In particular for $k=1$, we have $|\nabla^0 P|$ constant and so
\begin{align*}\bigl|\|H\|_2^{2}\normdfonc{P}^2-\normdpoly{P}^2\bigr|&\leqslant m_1C(n)D\normdpoly{P}^2
\end{align*}
Let $B_k=\sup\Bigl\{\frac{|\|H\|_2^{2k}\normdfonc{P}^2-\normdpoly{P}^2|}{\normdpoly{P}^2} \ \mid \ P\in\hkr\setminus\{0\}\Bigr\}$. Then using that $\nabla^0P\in{\mathcal H}^{k-1}(\R^{n+1})$ and \eqref{normgrad0}, we get for $1\leqslant i\leqslant k$
\begin{align*}B_k\leqslant B_{k-1}+m_k\bigl(\|X\|_\infty\|H\|_2\bigr)^{2k-2}C(n)D\leqslant C(n)D\sigma_k\bigl(\|X\|_\infty\|H\|_2\bigr)^{2k-1}
\end{align*}
We conclude using Theorem \ref{diambound}.
\end{proof}


\section{Proof of Theorem \ref{maintheo}}\label{potm}

Under the assumption of Theorem \ref{maintheo} we can use Lemma \ref{nldz}, Theorem \ref{diambound} and Inequality \eqref{estiray} to improve the estimate in Lemma \ref{Ppresquortho}.
\begin{lemma}\label{Ppresquortho2}
Let $q>\max(4,n)$, $p>2$ and $A>0$ be some reals. There exist some constants $C=C(p,q,n)$, $\alpha=\alpha(q,n)$ and $\beta=\beta(q,n)$ such that for any isometrically immersed hypersurface $M$ of $\R^{n+1}$ satisfying $(P_{p,\varepsilon})$ and $\Vol M\|H\|_q^n\leqslant A$, and for any $P\in\hkm$, we have
$$\Bigl|\|H\|_2^{2k}\normdfonc{P}^2-\normdpoly{P}^2\Bigr|\leqslant \varepsilon^\alpha\sigma_k(CA^{\beta})^{2k}\normdpoly{P}^2.$$
\end{lemma}

This allows to prove the following estimate on $\Delta P$.
\begin{lemma}\label{almosteigenf}
Let $k$ be an integer such that $\varepsilon^\alpha\sigma_k(CA^{\beta})^{2k}\leqslant\frac{1}{2}$ and $P\in\hkm$, we have
$$\bigl\|\Delta P-\mu_k\|H\|_2^2P\bigr\|_2\leqslant \sqrt{m_k}\mu_k(CA^{\beta})^k\varepsilon^\alpha\|H\|_2^2\|P\|_2$$
\end{lemma}

\begin{proof} From Formula \eqref{laplap}, we have
$$\|\Delta P-\mu_k\|H\|_2^2P\|_2\leqslant\|\mu_k(H^2-\|H\|_2^2)P\|_2+(n+2k-2)\|HdP(Z)\|_2+\|\nabla^0dP(Z,Z)\|_2$$
If $\varepsilon^{\alpha}\sigma_k (CA^{\beta})^{2k}\leqslant\frac{1}{2}$ we deduce from Lemma \ref{Ppresquortho2} that $\normdpoly{P}^2\leqslant 2\|H\|_2^{2k}\normdfonc{P}^2$. And using Lemma \ref{Pcarre} and Inequality \eqref{estiray}, we have
$$\displaylines{\|\mu_k(H^2-\|H\|_2^2)P\|_2^2\leqslant\frac{\mu^2_km_k}{\Vol M}\|P\|^2_{\S^n}\int_M|H^2-\|H\|_2^2|^2|X|^{2k}\vol\hfill\cr
\leqslant\frac{2\mu_k^2m_k}{\Vol M}\|P\|_2^2\bigl(\|H\|_2\|X\|_\infty\bigr)^{2k}\int_M(H^2-\|H\|_2^2)^2\vol\leqslant (CA^{\beta})^{2k}\varepsilon^\alpha\mu_k^2m_k\|H\|_2^4\|P\|_2^2}$$
where the last inequality comes from Inequality \ref{pinchcm} and the H\"older Inequality.
By technical Lemma of Section \ref{Homog}, we have
$$\displaylines{
\|HdP(Z)\|_2^2\leqslant\|P\|_{\S^n}^2\normdappli{(HZ)_k^*}^2\leqslant\|P\|_{\S^n}^2\frac{m_kk^2}{\Vol M}\|X\|_\infty^{2k-2}\int_M H^2|Z|^2\vol\hfill\cr\leqslant\varepsilon^\alpha (CA^{\beta})^{2k}k^2m_k\|H\|_2^4\|P\|_{2}^2\cr
\|\nabla^0dP(Z,Z)\|_2^2\leqslant\normdappli{(Z,Z)_k^{\star}}^2\|P\|_{\S^n}^2\leqslant m_k\alpha_{k,n}\|X\|_\infty^{2k-4}\|Z\|_4^4\|P\|_{\S^n}^2\hfill\cr
\leqslant  (CA^{\beta})^{2k}\varepsilon^\alpha m_k\alpha_{k,n}\|H\|_2^4\|P\|_{2}^2
}$$
which gives the result.
\end{proof}

Let $\nu>0$ and $E_k^\nu$ be the space spanned by the eigenfunctions of $M$ associated to an eigenvalue in the interval $\bigl[(1-\varepsilon^\alpha\sqrt{m_k}C^{k}-\nu)\|H\|_2^2\mu_k,(1+\varepsilon^\alpha\sqrt{m_k}C^{k}+\nu)\|H\|_2^2\mu_k\bigr]$. If $\dim E_k^\nu<m_k$, then there exists $P\in\hkm\setminus\{0\}$ which is $L^2$-orthogonal to $E_k^{\nu}$. Let $\displaystyle P=\sum_i f_i$ be the decomposition of $P$ in the Hilbert basis given by the eigenfunctions $f_i$ of $M$ associated respectively to $\lambda_i$. Putting $N:=\{i\ \ | \ \ f_i\notin E_k^{\nu}\}$, by assumption on $P$ we have 
$$\displaylines{(C^k\sqrt{m_k}\varepsilon^\alpha+\nu)^2\|H\|_2^4\mu_k^2\|P\|_2^2\leqslant\sum_{i\in N}\bigl(\lambda_i-\|H\|^2_2\mu_k\bigr)^2\|f_i\|_2^2=\|\Delta P-\mu_k\|H\|_2^2P\|_2^2\hfill\cr
\leqslant \mu_k^2C^{2k}m_k\|H\|_2^4\varepsilon^{2\alpha}\|P\|_2^2}$$
which gives a contradiction. We then have $\dim E_k^\nu\geqslant m_k$. We get the result by letting $\nu$ tends to $0$.


\section{Proof of Inequality \ref{passepartout}}\label{poipa}

We can assume $\eta\leqslant1$ and $\|H\|_2=1$ by a homogeneity argument. Let $x\in\S^n$ and set $V^n(s)=\Vol (B(x,s)\cap \S^n)$. Let $\beta>0$ small enough so that $(1+\eta/2)V^n\bigl((1+2\beta)r\bigr)\leqslant (1+\eta)V^n(r)$ and $(1-\eta/2)V^n\bigl((1-2\beta)r\bigr)\geqslant(1-\eta) V^n(r)$. Let $f_1:\S^n\to[0,1]$ (resp. $f_2:\S^n\to[0,1]$) be a smooth function such that $f_1=1$ on $B\bigl(x,(1+\beta)r\bigr)\cap\S^n$ (resp. $f_2=1$ on $B\bigl(x,(1-2\beta)r\bigr)\cap\S^n$) and $f_1=0$ outside $B\bigl(x,(1+2\beta) r\bigr)\cap\S^n$ (resp. $f_2=0$ outside $B\bigl(x,(1-\beta) r\bigr)\cap\S^n$). There exists a family $(P^i_k)_{k\leqslant N}$ such that $P^i_k\in\hkr$ and $A=\sup_{\S^n}\bigl|f_i-\sum_{k\leqslant N}P^i_k\bigr|\leqslant \|f_i\|_{\S^n}^2\eta/18$. We extend $f_i$ to $\R^{n+1}\setminus\{0\}$ by $f_i(X)=f_i\bigl(\frac{X}{|X|}\bigr)$. Then we have
$$\displaylines{
\Bigl|\|f_i\|_2^2-\frac{1}{\Vol\S^n}\int_{\S^n}|f_i|^2\Bigr|\leqslant I_1+I_2+I_3}$$
where 
$$I_1:=\Bigl|\frac{1}{\Vol M}\int_M\Bigl(|f_i|^2-\bigl(\sum_{k\leqslant N}|X|^{-k}P^i_k\bigr)^2\Bigr)\vol\Bigr|$$
$$I_2:=\Bigl|\frac{1}{\Vol M}\int_M\bigl(\sum_{k\leqslant N}|X|^{-k}P^i_k\bigr)^2\vol-\sum_{k\leqslant N}\|P^i_k\|_{\S^n}^2\Bigr|$$
and
$$I_3:=\Bigl|\frac{1}{\Vol\S^n}\int_{\S^n}\Bigl(\bigl(\sum_{k\leqslant N}P^i_k\bigr)^2-f_i^2\Bigr)\Bigr|.$$
On $\S^n$ we have $\bigl|f_i^2-(\sum_{k\leqslant N}P^i_k)^2\bigr|\leqslant A\bigl(2\sup_{\S^n}|f_i|+A\bigr)\leqslant\|f_i\|_{\S^n}^2\eta/6$ and on $M$ we have
\begin{align*}\Bigl|f_i^2(X)-\bigl(\sum_{k\leqslant N}|X|^{-k}P^i_k(X)\bigr)^2\Bigr|&=\Bigl|f_i^2\bigl(\frac{X}{|X|}\bigr)-\Bigl(\sum_{k\leqslant N}P^i_k\bigl(\frac{X}{|X|}\bigr)\Bigr)^2\Bigr|\leqslant\|f_i\|_{\S^n}^2\eta/6\cr.
\end{align*}
Hence $I_1+I_3\leqslant\|f_i\|_{\S^n}^2\eta/3$. Now
\begin{align*}I_2&\leqslant\left|\frac{1}{\Vol M}\insm\sum_{k\leqslant N}\frac{(P_k^i)^2}{|X|^{2k}}\vol-\sum_{k\leqslant N}\|P^i_k\|_{\S^n}^2\right|+\frac{1}{\Vol M}\left|\insm\sum_{1\leqslant k\neq k'\leqslant N}\frac{P_k^i P_{k'}^i}{|X|^{k+k'}}\vol\right|\displaybreak[2]\\
&\leqslant\frac{1}{\Vol M}\insm\sum_{k\leqslant N}\left|\frac{1}{|X|^{2k}}-\|H\|_2^{2k}\right|(P_k^i)^2\vol \\
&+\frac{1}{\Vol M}\insm\sum_{1\leqslant k\neq k'\leqslant N}\left|\frac{1}{|X|^{k+k'}}-\|H\|_2^{k+k'}\right||P_k^i P_{k'}^i|\vol\displaybreak[2]\\
&+\sum_{k\leqslant N}\left|\|H\|_2^{2k}\normdfonc{P_k^i}^2-\normdpoly{P_k^i}^2\right|+\sum_{1\leqslant k\neq k'\leqslant N}\frac{\|H\|_2^{k+k'}}{\Vol M}\left|\insm P_k^i P_{k'}^i\vol\right|
\end{align*}
From (\ref{estiray}) we have $\left|\frac{1}{|X|^{k+k'}}-\|H\|_2^{k+k'}\right|\leqslant N C^N\varepsilon^{\alpha}\|H\|_2^{k+k'}$. From this and Lemma \ref{Ppresquortho2}, we have
\begin{align*}I_2\leqslant N^2C^N\varepsilon^{\alpha}\sum_{k\leqslant N}\|H\|_2^{2k}\normdfonc{P_k^i}^2+\varepsilon^{\alpha}\sum_{k\leqslant N}\sigma_k C^{2k}\normdpoly{P_k^i}^2+\sum_{1\leqslant k\neq k'\leqslant N}\frac{\|H\|_2^{k+k'}}{\Vol M}\left|\insm P_k^i P_{k'}^i\vol\right|
\end{align*}
and, by Lemma \ref{almosteigenf}, we have
$$\displaylines{\Bigl|\frac{\|H\|_2^2(\mu_k-\mu_{k'})}{\Vol M}\int_M P^i_kP^i_{k'}\vol\Bigr|\leqslant\int_M\frac{|P^i_k(\Delta P^i_{k'}-\|H\|_2^2\mu_{k'} P^i_{k'})|}{\Vol M}\vol\hfill\cr
\hfill+\int_M\frac{|P^i_{k'}(\Delta P^i_k-\|H\|_2^2\mu_k P^i_k)|}{\Vol M}\vol
\hfill\cr
\hfill\leqslant \|P^i_k\|_2\bigl\|\Delta P^i_{k'}-\|H\|_2^2\mu_{k'} P^i_{k'}\bigr\|_2+\|P^i_{k'}\|_2\bigl\|\Delta P^i_k-\|H\|_2^2\mu_k P^i_k\bigr\|_2\cr
\leqslant 2\sqrt{m_N}\mu_NC^N\varepsilon^\alpha\|H\|_2^2\|P^i_{k'}\|_2\|P^i_k\|_2}$$
under the condition $\varepsilon^\alpha\sigma_N C^{2N}\leqslant\frac{1}{2}$. Since $\mu_k-\mu_{k'}\geqslant n$ when $k\neq k'$, we have 
$$\displaylines{\Bigl|\frac{1}{\Vol M}\int_M P^i_kP^i_{k'}\vol\Bigr|\leqslant \frac{2}{n}\sqrt{m_N}\mu_NC^N\varepsilon^\alpha\|P^i_{k'}\|_2\|P^i_k\|_2}$$
hence
\begin{align*}I_2&\leqslant N^2 C^N\varepsilon^{\alpha}\sum_{k\leqslant N}\|H\|_2^{2k}\normdfonc{P_k^i}^2+\varepsilon^{\alpha}\sum_{k\leqslant N}\sigma_k C^{2k}\normdpoly{P_k^i}^2\\
&+\frac{2}{n}\sqrt{m_N}\mu_NC^N\varepsilon^\alpha\sum_{1\leqslant k\neq k'\leqslant N}\|H\|_2^{k+k'}\|P^i_{k'}\|_2\|P^i_k\|_2\\
&\leqslant D_N\varepsilon^{\alpha}\sum_{k\leqslant N}\|H\|_2^{2k}\normdfonc{P_k^i}^2+\varepsilon^{\alpha}\sum_{k\leqslant N}\sigma_k C^{2k}\normdpoly{P_k^i}^2\\
&\leqslant\varepsilon^{\alpha}\sum_{k\leqslant N}(D_N(1+\varepsilon^{\alpha}\sigma_k C^{2k})+\sigma_k C^{2k})\normdpoly{P_k^i}^2\leqslant D'_N\varepsilon^{\alpha}\end{align*}
Where we have used the fact that $\displaystyle\bigl\|\sum_{k\leqslant N}P_k^i\bigr\|_{\S^n}^2$  is bounded by a constant. We infer that if\  $\varepsilon^\alpha\leqslant\frac{V^n((1-2\beta)r)\eta}{6D'_N\Vol\S^n}\leqslant\frac{\|f_i\|_{\S^n}^2\eta}{6D'_N}$, then we have
$$\Bigl|\|f_i\|_2^2-\frac{1}{\Vol\S^n}\int_{\S^n}|f_i|^2\Bigr|\leqslant\eta\|f_i\|^2_{\S^n}/2$$
Note that $N$ depends on $r$ and $\beta$ but not on $x$ since $O(n+1)$ acts transitively on $\S^n$.
Eventually, by assumption on $f_1$ and $f_2$ and by estimate \eqref{estiray}, we have
$$\displaylines{\frac{\Vol\left( B(x,(1+\beta)r-C\varepsilon^{\alpha})\cap X(M)\right)}{\Vol M}\leqslant\|f_1\|_2^2\leqslant(1+\eta/2)\|f_1\|_{\S^n}^2\hfill\cr\hfill
\leqslant(1+\eta/2)\frac{V^n((1+2\beta) r)}{\Vol\S^n}\leqslant(1+\eta)\frac{V^n( r)}{\Vol\S^n}\cr
\frac{\Vol\left( B(x,(1-\beta)r+C\varepsilon^{\alpha})\cap X(M)\right)}{\Vol M}\geqslant\|f_2\|_2^2\geqslant(1-\eta/2)\|f_2\|_{\S^n}^2\hfill\cr\hfill
\geqslant(1-\eta/2)\frac{V^n((1-2\beta) r)}{\Vol\S^n}\geqslant(1-\eta)\frac{V^n( r)}{\Vol\S^n}}$$
And by choosing $\varepsilon^{\alpha}=\min\left(\frac{\beta r}{C},\frac{V^n((1-2\beta)r)\eta}{6D'_N\Vol\S^n)}\right)$ we get
$$\left|\frac{\Vol(B(x,r)\cap X(M))}{\Vol M}-\frac{V^n(r)}{\Vol \S^n}\right|\leqslant\eta\frac{V^n(r)}{\Vol \S^n}$$

\section{Some examples}\label{se}
We set $\ieps=[\varepsilon,\frac{\pi}{2}]$ for $\varepsilon>0$ and let $\varphi : \ieps\longrightarrow (-1,+\infty)$ be a function continuous on $\ieps$ and smooth on $(\varepsilon,\frac{\pi}{2}]$. For any $0\leqslant k\leqslant n-2$, we consider the map
$$\begin{array}{rcl}\Phi_\varphi: \S^{n-k-1}\times\S^{k}\times \ieps &\longrightarrow &\R^{n+1}=\R^{n-k}\oplus\R^{k+1}\\
x=(y,z,r) & \longmapsto &  (1+\varphi(r))\bigl((\sin r) y+(\cos r) z\bigr)\end{array}$$
which is an embedding onto a manifold  $X_\varphi\subset\R^{n+1}$.
We denote respectively by $B(\varphi)$ and $H(\varphi)$ the second fundamental form and the mean curvature of $X_\varphi$. We have 

\begin{lemma}\label{courbmoy} Let $x=(y,z,r)\in\S^{n-k-1}\times\S^{k}\times \ieps $, $q=\Phi_\varphi(x)$ and $(u,v,h)\in T_x \xeps$. Then we have
$$\displaylines{
nH_{q}(\varphi)=\bigl( \coeff\bigr)^{-3/2}\Bigl[-(1+\varphi(r))\varphi''(r)+(1+\varphi(r))^2+2\varphi'^2(r)\Bigr]\hfill\cr
\hfill+\frac{\bigl( \coeff\bigr)^{-1/2}}{1+\varphi(r)}\Bigl[ -(n-k-1)\varphi'(r)\cot r+(n-1)(1+\varphi(r))+k\varphi'(r)\tan r\Bigr]\cr
\cr
|B_{q}(\varphi)|=\hfill\cr
\hfill\frac{(1+\varphi(r))^{-1}}{\bigl(1+(\frac{\varphi'(r)}{1+\varphi(r)})^2\bigr)^{1/2}}\max\Bigl(\bigl|1-\frac{\varphi'}{1+\varphi}\cot r\bigr|,\bigl|1+\frac{\varphi'}{1+\varphi}\tan r\bigr|,\bigl|1+\frac{(\varphi')^2-(1+\varphi)\varphi''}{\coeff}\bigr|\Bigl)}$$
\end{lemma}

\begin{proof}
Let $(u,v,h)\in T_x S_{\varepsilon}$ and put $w=d(\Phi_{\varphi})_x (u,v,h)\in T_q X_{\varphi}$ where $S_{\varepsilon}=\S^{n-k-1}\times\S^{k}\times \ieps$. An easy computation shows that 
\begin{align}\label{difphi}w&=(1+\varphi(r))((\sin r) u+(\cos r) v)\notag\\
&+\varphi'(r)((\sin r) y+(\cos r) z)h+(1+\varphi(r))((\cos r)y-(\sin r)z)h\end{align}
We set 
$$\tilde{N}_q=-\varphi'(r)((\cos r) y-(\sin r) z)+(1+\varphi(r))((\sin r) y+(\cos r) z)$$
and $\displaystyle N_q=\frac{\tilde{N}_q}{\left(\coeff\right)^{1/2}}$ is a unit normal vector field on $X_{\varphi}$.
Then we have
\begin{align}\label{defsec} B_q(\varphi)(w,w)&=\left\langle\nabla_{w}^0 N,w\right\rangle\notag=\left( \coeff\right)^{-1/2}\bigl\langle\nabla_{w}^0\tilde{N},w\bigr\rangle\notag\\
&=\left( \coeff\right)^{-1/2}\Bigl\langle\sum_{i=1}^{n+1} w(\tilde{N}^i)\partial_i,w\Bigr\rangle\end{align}
where $(\partial_i)_{1\leqslant i\leqslant n+1}$ is the canonical basis of $\R^{n+1}$. A straightforward computation shows that
\begin{align*}\sum_{i=1}^{n+1}w(\tilde{N}^i)\partial_i=&-\varphi'(r)((\cos r) u-(\sin r) v)+(1+\varphi(r))((\sin r) u+(\cos r) v)\\
&-\varphi''(r)((\cos r) y-(\sin r) z) h+2\varphi'(r)((\sin r) y+(\cos r) z)h\\
&+(1+\varphi(r))((\cos r) y-(\sin r) z) h\end{align*}
Reporting this in (\ref{defsec}) and using (\ref{difphi}) we get 
$$\displaylines{B_q(\varphi)((u,v,h),(u,v,h))=\frac{1}{\sqrt{\coeff}}\Bigl[-\varphi'(r)\bigl(1+\varphi(r)\bigr)\sin r \cos r (|u|^2-|v|^2)\cr
\hfill+(1+\varphi(r))^2 (\sin^2 r |u|^2 +\cos^2 r |v|^2) -\bigl(1+\varphi(r)\bigr)\varphi''(r)h^2+2\varphi'^2(r)h^2+(1+\varphi(r))^2h^2\Bigr]}$$ 
Now let $(u_i)_{1\leqslant i\leqslant n-k-1}$ and $(v_i)_{1\leqslant i\leqslant k}$ be orthonormal bases of respectively $\S^{n-k-1}$ at $y$ and $\S^{k}$ at $z$. We set $g=\Phi_{\varphi}^{\star} can$ and $\xi=(0,0,1)$, then we have
$$\displaylines{
\hfill g(u_i,u_j)=(1+\varphi(r))^2\sin^2 r\delta_{ij},\hfill g(v_i,v_j)=(1+\varphi(r))^2\cos^2 r\delta_{ij}, \hfill g(u_i,v_j)=0,\hfill\cr
\hfill g(\xi,\xi)=\coeff,\hfill g(u_i,\xi)=g(v_j,\xi)=0.\hfill}$$
Now setting $\tilde{u}_i=d(\Phi_{\varphi})_x(u_i)$, $\tilde{v}_i=d(\Phi_{\varphi})_x(u_i)$ and $\tilde{\xi}=d(\Phi_{\varphi})_x(\xi)$, the relation above allows us to compute the trace and norm
\begin{align*}
&|B_{q}(\varphi)|=\max\Bigl(\max_{i}\frac{|B_q(\varphi)(\tilde{u}_i,\tilde{u}_i)|}{g(u_i,u_i)},\max_j\frac{|B_q(\varphi)(\tilde{v}_j,\tilde{v}_j)|}{g(v_j,v_j)},\frac{|B_q(\varphi)(\tilde{\xi},\tilde{\xi})|}{g(\xi,\xi)}\Bigl)\\
&{=}\frac{1}{\sqrt{\coeff}}\max\Bigl(\bigl|1{-}\frac{\varphi '}{1{+}\varphi}\cot r\bigr|,\bigl|1+\frac{\varphi '}{1{+}\varphi}\tan r\bigr|,\bigl|1{+}\frac{(\varphi')^2-(1+\varphi)\varphi''}{\coeff}\bigr|\Bigr)\end{align*}
of the second fundamental form. 
\end{proof}

To prove Theorem \ref{ctrexple1}, we set $a<\frac{\pi}{10}$ and define the function $\fieps$ on $I_\varepsilon$ by
$$\fieps(r)=\left\{\begin{array}{ll}f_{\varepsilon}(r)=\varepsilon\displaystyle\int_1^{\frac{r}{\varepsilon}}\frac{dt}{\sqrt{t^{2(n-k-1)}-1}} & \text{if } \varepsilon\leqslant r\leqslant a+\varepsilon,\\[3mm]
u_{\varepsilon}(r) & \text{if } r\geqslant a+\varepsilon,\\[2mm]
b_\varepsilon& \text{if }r\geqslant 2a+\varepsilon,
\end{array}\right.$$
where $u_{\varepsilon}$ is chosen such that $\fieps$ is smooth on $(\varepsilon,\frac{\pi}{2}]$ and strictly concave on $(\varepsilon,2a+\varepsilon]$, and $b_\varepsilon$ is a constant. We have $f_{\varepsilon}(a+\varepsilon)\to 0$, $f'_{\varepsilon}(a+\varepsilon)\to 0$, $f''_{\varepsilon}(a+\varepsilon)\to 0$ and so $b_{\varepsilon}\to 0$ as $\varepsilon\to0$. Hence $b_\varepsilon$ can be chosen less than $\frac{1}{2}$ and $u_\varepsilon$ can be chosen such that $\varphi_\varepsilon$ tends uniformly on $I_\varepsilon$ and $\varphi_\varepsilon'\to0$, $\varphi_\varepsilon''\to0$ uniformly on any compact of $(\varepsilon,\frac{\pi}{2}]$.

Note that $\varphi_\varepsilon$ satisfies
\begin{align}\label{equadif}\varphi_\varepsilon''=-\frac{(n-k-1)(1+\varphi_\varepsilon'^2)}{r}\varphi_\varepsilon'\ \ \mbox{ on }\ (\varepsilon,a+\varepsilon].\end{align}
Moreover we have $\fieps(\varepsilon)=0$ and $\displaystyle\lim_{t\to\varepsilon}\varphi_\varepsilon'(t)=+\infty=-\lim_{t\to\varepsilon}\varphi_\varepsilon''(t)$. Moreover, we can define on $(-b_{\varepsilon},b_{\varepsilon})$ an application $\fiteps$ so that $\fiteps(t)=\fieps^{-1}(t)$ on  $[0,b_{\varepsilon})$ and $\fiteps(-t)=\fiteps(t)$.

Now let us consider the two applications $\Phi_{\varphi_\varepsilon}$ and $\Phi_{-\varphi_\varepsilon}$ defined as above, and put $\mpeps=X_{\varphi_\varepsilon}$ and  $\mmeps=X_{-\varphi_\varepsilon}$. Since $\fiteps$ satisfies the equation $yy''=(n-k-1)\bigl(1+(y')^2\bigr)$ with initial data $\fiteps(0)=\varepsilon$ and $\fiteps'(0)=0$, it is smooth at $0$, hence on $(-b_\varepsilon,b_\varepsilon)$, and so $\meps^k=\mpeps\cup\mmeps$ is a smooth submanifold of $\R^{n+1}$. Indeed, the function $F_{\varepsilon}(p_1,p_2)= |p_1|^2-|p|^2\sin^2\bigl(\fiteps(|p|-1)\bigr)$, defined on 
$$U=\{p=(p_1,p_2)\in\R^{n-k}\oplus\R^{k+1}/\,p_1\neq 0,\, p_2\neq 0,\, -b_{\varepsilon}+1<|p|<b_{\varepsilon}+1\}$$ is a smooth, local equation of $M^k_\varepsilon$ at the neighborhood of $\mpeps\cap\mmeps$ which satisfies
$$\nabla F_{\varepsilon}(p_1,p_2)=2p_1\cos^2\varepsilon-2p_2\sin^2\varepsilon\neq 0$$
on  $\mpeps\cap\mmeps$.

We denote respectively by $\Heps$ and $\beps$, the mean curvature and the second fundamental form of $\meps^k$. 

\begin{theorem}\label{ctrexple2} $\|\Heps\|_{\infty}^{~}$ and $\|B_\varepsilon\|_{n-k}^{~}$ remain bounded whereas $\|\Heps-1\|_1\to 0$ and $\bigl\||X|-1\bigr\|_\infty\to0$ when $\varepsilon\to 0$.
\end{theorem}

\begin{remark}
We have $\|B_\varepsilon\|_q\to\infty$ when $\varepsilon \to 0$, for any $q>n-k$.
\end{remark}

\begin{proof} From the lemma \ref{courbmoy} and the definition of $\fieps$, $H_\varepsilon$ and $|B_\varepsilon|$ converge uniformly to $1$ on any compact of $M_\varepsilon^k\setminus M_\varepsilon^+\cap M_\varepsilon^-$. On the neighborhood of $ M_\varepsilon^+\cap M_\varepsilon^-$, we have $n(\Heps)_x=n\heps(r)$ and $n \heps\leqslant\hepsi+\hepsd+\hepst$, where
$$\displaylines{
\hepsd(r)=k\frac{(\fieps'^2+(1\pm\fieps)^2)^{-1/2}}{1\pm\fieps}\fieps'\tan(r)\leqslant\frac{k}{1-b_\varepsilon}\tan\frac{\pi}{5}\hfill\cr
\hepst(r)=(n-1)(\fieps'^2+(1\pm\fieps)^2)^{-1/2}\hfill\cr
\hfill
+(\fieps'^2+(1\pm\fieps)^2)^{-3/2}((1\pm\fieps)^2+2\fieps'^2)\leqslant \frac{n+1}{1-b_\varepsilon}\hfill
}$$
and by differential Equation \eqref{equadif} we have
$$\displaylines{\hepsi(r)=\Bigl|(n-k-1)\frac{(\fieps'^2+(1\pm\fieps)^2)^{-1/2}}{1\pm\fieps}\fieps'\cot(r)+(\fieps'^2+(1\pm\fieps)^2)^{-3/2}(1\pm\fieps)\fieps''\Bigr|\cr
\leqslant(n-k-1)\frac{(\fieps'^2+(1\pm\fieps)^2)^{-1/2}}{1\pm\fieps}\fieps'\Bigl|\cot(r)-\frac{1}{r}\Bigr|\hfill\cr
+\frac{n-k-1}{r}\Bigl|\frac{(\fieps'^2+(1\pm\fieps)^2)^{-1/2}}{1\pm\fieps}\fieps'-(\fieps'^2+(1\pm\fieps)^2)^{-3/2} (1\pm\fieps)(1+\fieps'^2)\fieps'\Bigr|\cr
\leqslant\frac{n}{1-b_\varepsilon}\Bigl(\frac{1}{r}-\cot(r)\Bigr)\hfill\cr
+\frac{n\left(\fieps'^2+(1\pm\fieps)^2\right)^{-3/2}}{r(1\pm\fieps)}\fieps'\Bigl|\fieps'^2+(1\pm\fieps)^2-(1\pm\fieps)^2(1+\fieps'^2)\Bigr|\cr
\leqslant \frac{n}{1-b_\varepsilon}\Bigl(\frac{1}{r}-\cot(r)\Bigr)+\frac{n}{r}\fieps\frac{2\pm\fieps}{1\pm\fieps}\frac{\fieps'^3}{[\fieps'^2+(1\pm\fieps)^2]^{3/2}}\hfill\cr
\leqslant \frac{n}{1-b_\varepsilon}\Bigl(\frac{1}{r}-\cot(r)\Bigr)+\frac{n}{r}\fieps\frac{2+b_\varepsilon}{1-b\varepsilon}\hfill}$$
Since $\displaystyle\frac{\fieps}{r}=\frac{\varepsilon}{r}\int_1^{r/\varepsilon}\frac{dt}{\sqrt{t^{2(n-k-1)}-1}}\leqslant\frac{r}{\varepsilon}\int_1^{r/\varepsilon}\frac{dt}{\sqrt{t^{2}-1}}$ and $\frac{1}{x}\int_1^x\frac{dt}{\sqrt{t^2-1}}\sim_{+\infty}\frac{\ln x}{x}$, we get that $\hepsi$ is bounded on $M_\varepsilon^k$, hence $H_\varepsilon$ is bounded on $M_\varepsilon$. By the Lebesgue theorem we have $\|H_\varepsilon-1\|_1\to 0$. 

We now bound $\|B_\varepsilon\|_q$ with $q=n-k$. The volume element at the neighbourhood of $M^+_\varepsilon\cap M^-_\varepsilon$ is 
\begin{equation}
  \label{vol}
  \voleps=(1\pm\fieps)^{n}(1+(\frac{\fieps'}{1\pm\fieps})^2)^{1/2}\sin^{n-k-1}(r)\cos^k(r) dv_{n-k-1}dv_k dr
\end{equation}
where $dv_{n-k-1}$ and $dv_k$ are the canonical volume element of $\S^{n-k-1}$ and $\S^k$ respectively. By Lemma \ref{courbmoy} and Equation \eqref{equadif}, we have 
$$\displaylines{|B_\varepsilon|^q\voleps=\frac{1}{{(\coef)}^\frac{q}{2}}\max\Bigl(\bigl|1-\frac{\fieps'}{1\pm\fieps}\cot r\bigr|, \bigl|1+\frac{\fieps'}{1\pm\fieps}\tan r\bigr|,\hfill\cr
\hfill\bigl|1+\frac{\fieps'^2+(n-k-1)(1\pm\fieps)(1+\fieps'^2)\fieps'/r}{\coef}\bigr|\Bigr)\Bigr]^q \voleps}$$
Noting that $\frac{x}{\sqrt{1+x^2}}\leqslant\min(1,x)$, it is easy to see that, if we set $h_\varepsilon=\min(1,|\varphi'_\varepsilon|)$ 
\begin{align*}\frac{\bigl|1-\frac{\fieps'}{1\pm\fieps}\cot r\bigr|}{\sqrt{\coef}}&\leqslant\frac{1}{\sqrt{\fieps'^2+(1\pm\fieps)^2}}+\frac{\frac{\fieps'}{1\pm\fieps}}{\sqrt{\frac{\fieps'^2}{(1\pm\fieps)^2}+1}}\frac{\cot r}{1\pm\fieps}\\
 &\leqslant\frac{1}{1-\fieps}+\frac{h_{\varepsilon}\cot r}{(1-\fieps)^2}\\
&\leqslant4\left(1+\frac{h_{\varepsilon}}{r}\right)
\end{align*}
Similarly for $r\in[\varepsilon, \pi/5+\varepsilon]$ and $\varepsilon$ small enough, we have
$$\frac{\bigl|1+\frac{\fieps'}{1\pm\fieps}\tan r\bigr|}{\sqrt{\coef}}\leqslant 4(1+h_\varepsilon\tan r)\leqslant 8(1+h_\varepsilon r)\leqslant 8\bigl(1+\frac{h_\varepsilon}{r}\bigr)$$
And since $\fieps'=0$ for $r\geqslant\pi/5+\varepsilon$, this inequality is also true for $r\in(\varepsilon,\pi/2]$. Moreover
\begin{align*}&\frac{1}{\sqrt{\coef}}\Bigl|1+\frac{\fieps'^2+(n-k-1)(1\pm\fieps)(1+\fieps'^2)\fieps'/r}{\coef}\Bigr|\\
&\leqslant\frac{1}{1\pm\fieps}+\frac{\fieps'^2}{(\coef)^{3/2}}+\frac{n}{r}\frac{(1\pm\fieps)(1+\fieps'^2)}{\coef}\frac{|\fieps'|}{(\coef)^{1/2}}\\
&\leqslant\frac{2}{1\pm\fieps}+\frac{nh_\varepsilon}{r(1-\fieps)}\frac{(1\pm\fieps)(1+\fieps'^2)}{\coef}\\
&\leqslant\frac{2}{1\pm\fieps}+2\frac{nh_\varepsilon}{r}\frac{(1+\fieps)^2}{(1-\fieps)^2}\\
&\leqslant 2\bigl(2+9\frac{nh_\varepsilon}{r}\bigr)
\end{align*}
It follows that
\begin{align*}
|B_{\varepsilon}|^q\voleps&\leqslant C(n,k)\bigl(1+\frac{h_\varepsilon}{r}\bigr)^q\voleps\\
&\leqslant C(n,k)(r+h_\varepsilon)^q r^{-1}\bigl(1+\frac{\fieps'}{1\pm\fieps}\bigr)dv_{n-k-1}dv_{k}dr\\
&\leqslant C(n,k)r^{-1}(r+h_\varepsilon)^q\Bigl(1+\frac{1}{\sqrt{(r/\varepsilon)^{2(n-k-1)}-1}}\Bigr)dv_{n-k-1}dv_kdr\end{align*}
Now
$$\displaylines{\int_{\meps^k}|B_\varepsilon|^q\voleps\leqslant C(n,k)\Bigl(\int_\varepsilon^{2^\frac{1}{2(n-k-1)}\varepsilon}r^{-1}\Bigl(1+\frac{1}{\sqrt{(r/\varepsilon)^{2(n-k-1)}-1}}\Bigr)dr\hfill\cr
\hfill+\int_{2^\frac{1}{2(n-k-1)}\varepsilon}^{2a+\varepsilon} r^{n-k-1}\Bigl(1+\frac{1}{r\sqrt{(r/\varepsilon)^{2(n-k-1)}-1}}\Bigr)^qdr\Bigr)\cr
\leqslant C(n,k)\Bigl(\int_1^{2^\frac{1}{2(n-k-1)}}s^{-1}\Bigl(1+\frac{1}{\sqrt{s^{2(n-k-1)}-1}}\Bigr)ds+\int_{2^\frac{1}{2(n-k-1)}}^{2a/\varepsilon+1} s^{n-k-1}\bigl(\varepsilon+\frac{1}{s^{q}}\bigr)^qds\Bigr)\cr\\}
$$
Since $\varepsilon^\frac{-1}{q}\leqslant\frac{2a}{\varepsilon}+1$ for $\varepsilon$ small enough we have
$$\displaylines{\int_{\meps^k}|B_\varepsilon|^q\voleps\leqslant C(n,k)\Bigl(1+\int_{2^\frac{1}{2(n-k-1)}}^{\varepsilon^\frac{-1}{q}} \frac{2s^{n-k-1}}{s^{q^2}}ds+\int_{\varepsilon^\frac{-1}{q}}^{2a/\varepsilon+1} 2s^{n-k-1}\varepsilon^qds\Bigr)\cr
\leqslant C(n,k)\bigl(1+\varepsilon^{n-k-1}\bigr)\\}$$
which remains bounded when $\varepsilon\to0$.
\end{proof}

Since $\varphi_\varepsilon$ is constant outside a neighborhood of $M_\varepsilon^+\cap M_\varepsilon^-$ (given by $a$), $M^k_\varepsilon$ is a smooth submanifold diffeomorphic to the sum of two spheres $\S^n$ along a (great) subsphere $\S^k\subset\S^n$.
\begin{center}
\includegraphics[width=1.5cm]{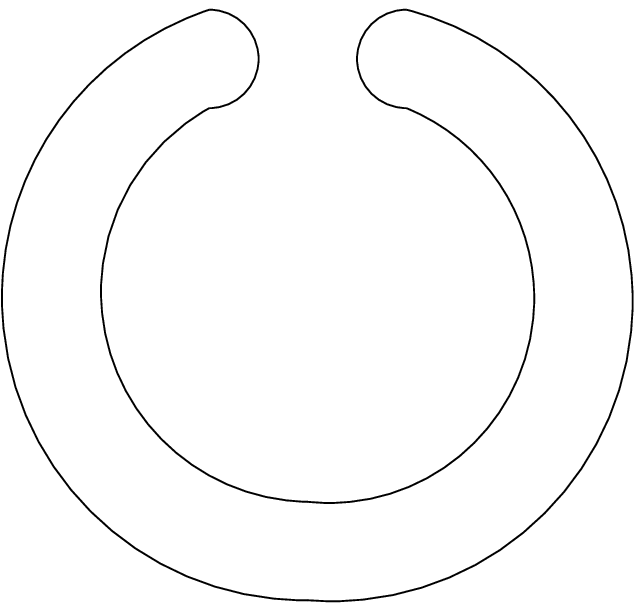}
\end{center}
If we denote $\tilde{M}^k_\varepsilon$ one connected component of the points of $M^k_\varepsilon$ corresponding to $r\leqslant 3a$, we get some pieces of hypersurfaces
\begin{center}
\includegraphics[width=1cm]{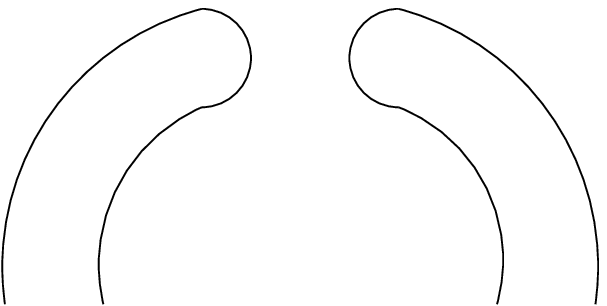}
\end{center}
that can be glued together along pieces of spheres of constant curvature to get a smooth submanifold $M_\varepsilon$, diffeomorphic to $p$ spheres $\S^n$ glued each other along $l$ subspheres $S_i$, and with curvature satisfying the bounds of Theorem \ref{ctrexple1} (when all the subspheres have dimension $0$) or of Remark \ref{rem}.
\begin{center}
\includegraphics[width=2cm]{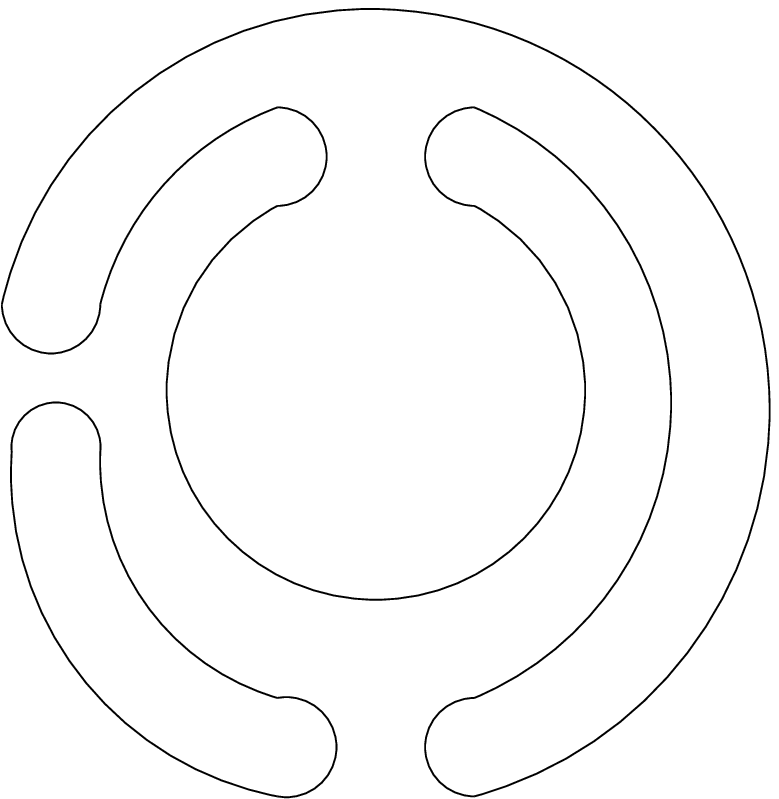}
\end{center}
Since the surgeries are performed along subsets of capacity zero, the manifold constructed have a spectrum close to the spectrum of $p$ disjoints spheres of radius close to $1$ (i.e. close to the spectrum of the standard $\S^n$ with all multiplicities multiplied by $p$).
More precisely, we set $\eta\in[2\varepsilon,\frac{\pi}{20}]$, and for any subsphere $S_i$, we set $N_{i,\eta,\varepsilon}$ the tubular neighborhood of radius $\eta$ of the submanifold $\tilde{S}_i=M_{\varepsilon,i}^+\cap M_{\varepsilon,i}^-$ in the local parametrization of $M_\varepsilon$ given by the map $\Phi_{\varphi_{\varepsilon,i}}$ associated to the subsphere $S_i$. We have $\meps=\Omega_{1,\eta,\varepsilon}\cup\cdots\cup \Omega_{p,\eta,\varepsilon}\cup N_{1,\eta, \varepsilon}\cup \cdots\cup N_{l,\eta,\varepsilon}$ where $\Omega_{i,\eta,\varepsilon}$ are the connected component of $M\setminus\cup_i N_{i,\eta,\varepsilon}$. The $\Omega_{i,\eta,\varepsilon}$ are diffeomorphic to some $S_{i,\eta}$ (which does not depend on $\varepsilon$ and $\eta$) open set of $\S^n$ which are complements of neighborhoods of subspheres of dimension less than $n-2$ and radius $\eta$, endowed with metrics which converge in $\mathcal{ C}^1$ topology to standard metrics of curvature $1$ on $S_{i,\eta}$. Indeed, $\varphi_\varepsilon$ converge to 0 in topology $\mathcal{ C}^2$ on $[r^{i,\pm}_{\varepsilon,\eta},\frac{\pi}{2}]$, where $\int_\varepsilon^{r_{\varepsilon,\eta}^{i,\pm}}\sqrt{(1\pm\varphi_{\varepsilon,i})^2+(\varphi_{\varepsilon,i}')^2}=\eta$ since it converges in $\mathcal{ C}^1$ topology on any compact of $[\varepsilon,\frac{\pi}{2}]$ and since we have
$$\displaylines{\eta\geqslant\int_\varepsilon^{r^{i,\pm}_{\varepsilon,\eta}}(1-b_{i,\varepsilon})\,dt=(r^{i,\pm}_{\varepsilon,\eta}-\varepsilon)(1-b_{i,\varepsilon})
\cr
\eta\leqslant\int_\varepsilon^{r^{i,\pm}_{\varepsilon,\eta}}(1+b_{i,\varepsilon})\,dt+\int_\varepsilon^{r^{i,\pm}_{\varepsilon,\eta}}\frac{dt}{\sqrt{(\frac{t}{\varepsilon})^{2(n-k-1)}-1}}=(r^{i,\pm}_{\varepsilon,\eta}-\varepsilon)(1+b_{i,\varepsilon})\cr\hspace{7cm}+\varepsilon\int_1^{+\infty}\frac{dt}{\sqrt{t^{2(n-k-1)}-1}}}$$
so $r_{\varepsilon,\eta}^\pm\to\eta$ when $\varepsilon\to 0$.
 So the spectrum of $\cup_i\Omega_{i,\eta,\varepsilon}\subset M_\varepsilon$ for the Dirichlet problem converges to the spectrum of $\amalg_i S_{i,\eta}\subset \amalg_i\S^n$ for the Dirichlet problem as $\varepsilon$ tends to $0$ (by the min-max principle). Since any subsphere of codimension at least $2$ has zero capacity in $\S^n$, we have that  the spectrum of $\amalg_i S_{i,\eta}\subset\amalg_i \S^n$ for the Dirichlet problem converges to the spectrum of $\amalg_i\S^n$ when $\eta$ tends to $0$ (see for instance \cite{Co} or adapt what follows). Since the spectrum of $\amalg_i\S^n$ is the spectrum of $\S^n$ with all multiplicities multiplied by $p$, by diagonal extraction we infer the existence of two sequences $(\varepsilon_m)$ and $(\eta_m)$ such that $\varepsilon_m\to0$, $\eta_m\to 0$ and  the spectrum of $\cup_i\Omega_{i,\eta_m,\varepsilon_m}\subset M_{\varepsilon_m}$ for the Dirichlet problem converges to the spectrum of $\S^n$ with all multiplicities multiplied by $p$.

Finally, note that $\lambda_{\sigma}(M_\varepsilon)\leqslant\lambda_\sigma(\cup_i\Omega_{i,2\eta,\varepsilon})$ for any $\sigma$ by the Dirichlet principle. On the other hand, by using functions of the distance to the $\tilde{S}_i$ we can easily construct on $M_\varepsilon$ a function $\psi_\varepsilon$ with value in $[0,1]$, support in $\cup_i\Omega_{i,\eta,\varepsilon}$, equal to $1$ on $\cup_i\Omega_{i,2\eta,\varepsilon}$ and whose gradient satisfies $|d \psi_\varepsilon|_{g_\varepsilon}^{~}\leqslant\frac{2}{\eta}$. It readily follows that
$$\|1-\psi_\varepsilon^2\|_1^{~}+\|d\psi_\varepsilon\|^{2}_2\leqslant(1+\frac{4}{\eta^2})\sum_i\frac{\Vol N_{i,2\eta,\varepsilon}}{\Vol M_\varepsilon}$$
To estimate $\sum_i\Vol N_{i,2\eta,\varepsilon}$, note that $N_{i,2\eta,\varepsilon}$ corresponds to the set of points with $r^{i,\pm}\leqslant r^{i,\pm}_{\varepsilon,2\eta}$ in the parametrization of $M_\varepsilon$ given by $\Phi_{\varphi_{\varepsilon,i}}$ at the neighborhood of $\tilde{S}_i$, where, as above, $r^{i,\pm}_{\varepsilon,2\eta}$ is given by
$$\int_\varepsilon^{r^{i,\pm}_{\varepsilon,2\eta}}\sqrt{(1\pm\varphi_{\epsilon,i})^2+(\varphi_{\epsilon,i}')^2}=2\eta$$
hence satisfies $\frac{1}{2}(r^{i,\pm}_{\varepsilon,2\eta}-\varepsilon)\leqslant 2\eta$ (since we have $1-\varphi_{\varepsilon,i}\geqslant\frac{1}{2}$).
By formula \ref{vol}, we have 
$$\displaylines{\Vol N_{i,2\eta,\varepsilon}\leqslant C(n)\int_\varepsilon^{r^-_\eta}(1-\varphi_{\varepsilon,i})^{n-1}\sqrt{(1-\varphi_{\varepsilon,i})^2+(\varphi_{\varepsilon,i}')^2}t^{n-k-1}dt\hfill\cr
\hfill+ C(n)\int_\varepsilon^{r^+_\eta}(1+\varphi_{\varepsilon,i})^{n-1}\sqrt{(1+\varphi_{\varepsilon,i})^2+(\varphi_{\varepsilon,i}')^2}t^{n-k-1}dt\hfill\cr
\hfill\leqslant C(n)(4\eta+\varepsilon)^{n-k-1}\eta\leqslant C(n,k)\eta^{n-k}}$$
where we have used that $\varphi_{\varepsilon,i}\leqslant2$ and $2\varepsilon\leqslant\eta$. We then have
$$\|1-\psi_\varepsilon^2\|_1^{~}+\|d\psi_\varepsilon\|_2^2\leqslant C(n,k,l,p)\eta^{n-k}$$
To end the proof of the fact that $M_{\varepsilon_m}$ has a spectrum close to that of $\cup_i\Omega_{i,\eta_m,\varepsilon_m}$ we need the following proposition, whose proof is a classical Moser iteration (we use the Sobolev Inequality \ref{simon}).
\begin{proposition}\label{norminffoncprop}  For any $q>n$ there exists a constant $C(q,n)$ so that if $(M^n,g)$ is any Riemannian manifold isometrically immersed in $\R^{n+1}$ and $E_N=\langle f_0,\cdots,f_N\rangle$is the space spanned by the eigenfunctions associated to $\lambda_0\leqslant\cdots\leqslant\lambda_N$, then for any $f\in E_N$ we have $$\|f\|_{\infty}\leqslant C(q,n)\left((\Vol M)^{1/n}(\lambda_N^{1/2}+\|H\|_q)\right)^{\gamma}\|f\|_2$$
where $\gamma=\frac{1}{2}\frac{qn}{q-n}$.  
\end{proposition}
Since we already know that $\lambda_\sigma(M_{\varepsilon_m})\leqslant\lambda_\sigma(\cup_i\Omega_{i,\eta_m,\varepsilon_m})\to\lambda_{E(\sigma/p)}(\S^n)$ for any $\sigma$ when $m\to\infty$, we infer that for any $N$ there exists $m=m(N)$ large enough such that on $M_{\varepsilon_m}$ and for any $f\in E_N$, we have (with $q=2n$ and since $\|H\|_\infty\leqslant C(n)$)
$$\|f\|_{\infty}\leqslant C(p,N,n)\|f\|_2$$
By the previous estimates, if we set
$$L_{\varepsilon_m}:f\in E_N\mapsto\psi_{\varepsilon_m} f\in H^1_0(\cup_i\Omega_{i,\eta_m,\varepsilon_m})$$
then we have
$$\|f\|_2^2\geqslant\|L_{\varepsilon_m}(f)\|_2^2\geqslant\|f\|_2^2-\|f\|_\infty^2\|1-\psi_{\varepsilon_m}^2\|_1\geqslant\|f\|_2^2\bigl(1-C(k,l,p,N,n)\eta_m^{n-k}\bigr)$$
and 
$$\displaylines{\|dL_{\varepsilon_m}(f)\|_2^2=\frac{1}{\Vol M_{\varepsilon_m}}\int_{M_{\varepsilon_m}}|fd\psi_{\varepsilon_m}+\psi_{\varepsilon_m} df|^2\hfill\cr
\leqslant(1+h)\|df\|^2_2+(1+\frac{1}{h})\frac{1}{\Vol M_{\varepsilon_m}}\int_{M_{\varepsilon_m}}f^2|d\psi_{\varepsilon_m}|^2\cr
\hfill\leqslant(1+h)\|df\|^2_2+(1+\frac{1}{h})C(k,l,p,N,n)\|f\|_2^2\eta_m^{n-k}}$$
for any $h>0$. We set $h=\eta_m^\frac{n-k}{2}$. For $m=m(k,l,p,N,n)$ large enough, $L_{\varepsilon_m}:E_N\to H^1_0(\cup_i\Omega_{i,\eta_m,\varepsilon_m})$ is injective and for any $f\in E_N$, we have
$$\frac{\|dL_{\varepsilon_m}(f)\|_2^2}{\|L_{\varepsilon_m}(f)\|_2^2}\leqslant (1+C(k,l,p,N,n)\eta_m^\frac{n-k}{2})\frac{\|df\|_2^2}{\|f\|_2^2}+C(k,l,p,N,n)\eta_m^\frac{n-k}{2}$$
By the min-max principle, we infer that for any $\sigma\leqslant N$, we have
$$\lambda_\sigma(M_{\varepsilon_m})\leqslant\lambda_\sigma(\cup_i\Omega_{i,\eta_m,\varepsilon_m})\leqslant(1+C(k,l,p,N,n)\eta_m^\frac{n-k}{2})\lambda_\sigma(M_{\varepsilon_m})+C(k,l,p,N,n)\eta_m^\frac{n-k}{2}$$
Since $\lambda_\sigma(\cup_i\Omega_{i,\eta_M,\varepsilon_m})\to \lambda_{E(\sigma/p)}(\S^n)$, this gives that $\lambda_\sigma(M_{\varepsilon_m})\to \lambda_{E(\sigma/p)}(\S^n)$ for any $\sigma\leqslant N$. By diagonal extraction we get the sequence of manifolds $(M_j)$ of Theorem \ref{ctrexple1}.

To construct the sequence of Theorem \ref{ctrexple3}, we consider the sequence of embedded submanifolds $(M_j)$ of Theorem \ref{ctrexple1} for $p=2$, $k=n-2$ and $l=1$. Each element of the sequence admits a covering of degree $d$ given by $y\mapsto y^d$ in the  local charts associated to the maps $\Phi$. We endow these covering with the pulled back metrics. Arguing as above, we get that the spectrum of the new sequence converge to the spectrum of two disjoint copies of 
$$\bigl(\S^1\times\S^{n-2}\times[0,\frac{\pi}{2}],dr^2+d^2\sin^2rg_{\S^1}+\cos^2rg_{\S^{n-2}}\bigr).$$


\begin{thebibliography}{99}
\bibitem{Au} {\sc E. Aubry}, {\em Pincement sur le spectre et le volume en courbure de Ricci positive},  Ann. Sci. �cole Norm. Sup. (4) {\bf 38} (2005), n�3, 387-405.
\bibitem{An} {\sc M.~Anderson}, {\em Metrics of positive Ricci curvature with large diameter}, Manuscripta Math. {\bf 68} (1990), p. 405--415.
\bibitem{BGM} {\sc M. Berger, P. Gauduchon, E. Mazet}, {\em Le spectre d'une vari\'et\'e riemannienne}, Lecture Notes in Math. {\bf 194}, Springer-Verlag, Berlin-New York (1971).
\bibitem{Ch-Co} {\sc J.~Cheeger,~T.~Colding}, {\em On the structure of spaces with Ricci curvature bounded below. I} J. Differential Geom. {\bf 46} (1997), p. 406--480.
\bibitem{colgros} {\sc B. Colbois, J.-F. Grosjean}, {\em A pinching theorem for the first eigenvalue of the Laplacian on hypersurfaces of the Euclidean space}, Comment. Math. Helv. {\bf 82}, (2007), 175-195.
\bibitem{Co1} {\sc T.~Colding},  {\em Shape of manifolds with positive Ricci curvature}, Invent. Math. {\bf 124} (1996), p. 175--191.
\bibitem{Co2} {\sc T.~Colding},  {\em Large manifolds with positive Ricci curvature}, Invent. Math. {\bf 124} (1996), p. 193--214.
\bibitem{Co} {\sc G.~Courtois},  {\em Spectrum of manifolds with holes}, J. Funct. Anal. {\bf 134} (1995), no.1, p. 194--221.
\bibitem{Cr} {\sc C. Croke}, {\em An eigenvalue pinching theorem},
\bibitem{HK} {\sc T. Hasanis, D. Koutroufiotis}, {\em Immersions of bounded mean curvature}, Arc. Math. {\bf 33}, (1979), 170-171.
\bibitem{Il} {\sc S.~Ilias}, {\em Constantes explicites pour les in�galit�s de Sobolev sur les vari�t�s riemanniennes compactes}, Ann. Inst. Fourier {\bf 33} (1983), p.151-165.
\bibitem{Mi-Si} {\sc J. H. Michael, L. M. Simon,}
{\em Sobolev and mean-value inequalities on generalized submanifolds of $R\sp{n}$}, Comm. Pure Appl. Math. {\bf 26} (1973), 361-379.
\bibitem{Ot} {\sc Y.~Otsu},  {\em On manifolds of positive Ricci curvature with large diameter}, Math. Z. {\bf 206} (1991), p. 252--264.
\bibitem{Pe} {\sc P.~Petersen}, {\em On eigenvalue pinching in positive Ricci curvature}, Invent. Math. {\bf 138} (1999), p. 1--21.
\bibitem{Re} {\sc R.C. Reilly}, {\em On the first eigenvalue of the Laplacian for compact submanifolds of Euclidean space}, Comment. Math. Helv. {\bf 52}, (1977), 525-533.
\bibitem{roth}{\sc J. Roth}, {\em Extrinsic radius pinching for hypersurfaces of space forms}, Diff. Geom. Appl. {\bf 25}, No 5, (2007) , 485-499.


\end{thebibliography}
\end{document}